\documentclass[11 pt]{article}  
\usepackage[utf8]{inputenc}
\usepackage{amsmath}
\usepackage{amsfonts}
\usepackage{amssymb}
\usepackage{graphicx}
\usepackage{mathrsfs}
\usepackage{upref,amsthm,amsxtra,exscale}
\usepackage{stmaryrd}
\usepackage{cite}
\usepackage[colorlinks=true,urlcolor=blue,
citecolor=red,linkcolor=blue,linktocpage,pdfpagelabels,bookmarksnumbered,bookmarksopen]{hyperref}

\newcommand\inter[1]{\llbracket #1\rrbracket}

\usepackage[cm]{fullpage}

\usepackage{subcaption}
\usepackage{caption}
\usepackage{cleveref}

\newtheorem{theorem}{Theorem}[section]

\newtheorem{remark}[theorem]{Remark}
\newtheorem{lemma}[theorem]{Lemma}
\newtheorem{proposition}[theorem]{Proposition}

\numberwithin{equation}{section}

\usepackage{enumitem}

\def\N{\mathbb{N}}

\def\dist{\operatorname{dist}}

\def\eps{\varepsilon}

\def\tilde{\widetilde}

\newcommand{\RR}{\mathbb{R}}

\newcommand{\weH}[1]{\mathbb H^{#1}(\Omega;\ell)}

\def\R{\mathbb{R}}

\def\mH{\mathbb{H}}
\def\cE{\mathcal{E}}

\def\d{\textnormal{d}}

\def\sideremark#1{\ifvmode\leavevmode\fi\vadjust{\vbox to0pt{\vss% the remark
 \hbox to 0pt{\hskip\hsize\hskip1em%                          will appear only
 \vbox{\hsize2.1cm\tiny\raggedright\pretolerance10000%          on the side
  \noindent #1\hfill}\hss}\vbox to15pt{\vfil}\vss}}}%

    \usepackage{color}
\usepackage[dvipsnames]{xcolor}
\usepackage{tikz}
\usepackage{pgfplots}
\usepackage{pgfplotstable}
\usetikzlibrary{positioning}
\usepackage{booktabs}

\title{FEM for 1D-problems involving the logarithmic Laplacian: error estimates and numerical implementation}

\author{V\'ictor Hern\'andez-Santamar\'ia\footnote{The work of V. Hern\'andez-Santamar\'ia is supported by the program ``Estancias Posdoctorales por México para la Formación y Consolidación de las y los Investigadores por México'' of SECIHTI (Mexico). He also received support from Project
A1-S-10457
and
CBF2023-2024-116
of SECIHTI and by UNAM-DGAPA-PAPIIT grants
IN102925, IN117525, IN104922, and IA100324 (Mexico).} \and
 Sven Jarohs
 \and
Alberto Salda\~{n}a\footnote{ A. Saldaña is supported  by  
SECIHTI grant CBF2023-2024-116 (Mexico) and by UNAM-DGAPA-PAPIIT grant IN102925 (Mexico).}\and
Leonard Sinsch
}

\date{}

\begin{document}

\maketitle

\abstract{
We present the numerical analysis of a finite element method (FEM) for one-dimensional Dirichlet problems involving the logarithmic Laplacian (the pseudo-differential operator that appears as a first-order expansion of the fractional Laplacian as the exponent $s\to 0^+$). Our analysis exhibits new phenomena in this setting; in particular, using recently obtained regularity results, we prove rigorous error estimates and provide a logarithmic order of convergence in the energy norm using  suitable \emph{log}-weighted spaces. Moreover, we show that the stiffness matrix of logarithmic problems can be obtained as the derivative of the fractional stiffness matrix evaluated at $s=0$.
Lastly, we investigate the relationship between the discrete eigenvalue problem and its convergence to the continuous one.
}

\bigskip
\bigskip

\noindent\textsc{Keywords:} Finite elements, zero-order kernel, quasi-interpolation, weighted logarithmic norms.
\medskip

\noindent\textsc{MSC2020:}
35S15 · % Boundary value problems for PDEs with pseudodifferential operators
65N15 · % Error bounds for boundary value problems involving PDEs
65N30 · % Finite element, Rayleigh-Ritz and Galerkin methods for boundary value problems involving PDEs
35B65 ·	%Smoothness and regularity of solutions to PDEs
65N25   %Numerical methods for eigenvalue problems for boundary value problems involving PDEs

\medskip

\section{Introduction}

In this paper we analyze and implement the finite element method (FEM) to one-dimensional nonlocal Dirichlet boundary value problems involving the logarithmic Laplacian such as
\begin{align}\label{P:intro}
    L_\Delta u = f \quad \text{ in }\Omega,\qquad u=0\quad \text{ on }\R\backslash \Omega,
\end{align}
where $\Omega$ is an open interval, $f$ belongs to a suitable (log-Hölder) space, and $L_\Delta$ denotes the logarithmic Laplacian given by
\begin{align}\label{LL}
     L_\Delta u(x)
     =
     \int_{x-1}^{x+1}\frac{u(x)-u(y)}{|x-y|}\, dy
     -\int_{\R\backslash(x-1,x+1)}\frac{u(y)}{|x-y|}\, dy
     +\rho_1\, u(x)\qquad \text{ for }x\in \Omega.
 \end{align}
 Here $\rho_1\approx-1.15$ is an explicit negative constant given in Section \ref{sec:notation}. The logarithmic Laplacian $L_\Delta$ is a pseudo-differential operator with Fourier symbol $2\ln|\xi|$, namely, ${\mathcal F}(L_\Delta \varphi)(\xi)=2\ln|\xi| {\mathcal F}(\varphi)(\xi)$ for all $\varphi\in C^\infty_c(\R),$ where ${\mathcal F}$ denotes the Fourier transform.  This operator appears naturally as a first order expansion of the (integral) fractional Laplacian $(-\Delta)^s$ (the pseudo-differential operator with Fourier symbol $|\xi|^{2s}$) as $s\to 0^+$; in particular, in \cite[Theorem 1.1]{CW19} it is shown that, for all $\varphi\in C^\infty_c(\R)$,
\begin{align}\label{intro:exp}
(-\Delta)^s\varphi = \varphi + sL_\Delta \varphi + o(s)\qquad \text{as $s\to 0^+$ in }L^p(\R) \text{ with }1<p\leq \infty.
\end{align}

 Problems such as \eqref{P:intro} play a key role in the understanding of fractional equations as $s\to 0$. This asymptotic analysis is relevant, both in applications and for theoretical reasons.  In applications, for instance, several phenomena that are modeled with a fractional-type diffusion are optimized in some sense for small values of $s$, see \cite{Caffarelli17} and the discussion and references in \cite{HSS22}.  On the other hand,  the mathematical structures that appear in the limit as $s\to 0^+$ for linear and nonlinear fractional problems are highly nontrivial and interesting; for instance, we refer to \cite{CW19,LW21,FJW22} for the analysis of linear and eigenvalue problems, to \cite{AS22,HSS22} for the analysis of some relevant nonlinear problems, and to \cite{CS22} for the regularity properties of solutions to \eqref{P:intro}, in particular, \cite[Theorem 1.1]{CS22} gives some conditions for the existence of classical solutions to \eqref{P:intro}.  We also mention \cite{JSW20} and \cite{JSW24}, where the operator $L_\Delta$ is used to characterize the derivative of the solution mapping of nonlocal problems involving the fractional Laplacian and to provide a sharp rate of convergence of the nonlocal-to-local transition as $s\to1$. See also \cite{chen2023taylor}, where higher-order expansions of the fractional Laplacian are considered.

The problem \eqref{P:intro} has a variational structure (see \cite{CW19}), which allows the use of powerful functional analysis tools to explore its solvability. Nevertheless, the operator $L_\Delta$ exhibits several pathologies that make its study difficult. Here is a short list of some of them:
\begin{enumerate}
    \item It lacks good scaling properties (see the discussion in \cite{CS22}).  In particular, scalings in this setting produce zero-order terms (see \cite[Appendix A.1]{HSLRS23}).
    \item For every bounded open set $U\subset \R$ there is $\lambda>0$ such that the problem \eqref{P:intro} does not have a solution for $f\equiv 1$ and $\Omega=\lambda U$, see \cite[Remark 5.9]{CS22}.
    \item The operator $L_\Delta$ does not satisfy the maximum principle in general. Nevertheless, positivity preserving properties do hold under additional assumptions on the size of the domain, see \cite[Theorem 1.8 and Corollary 1.9]{CW19}.
    \item Classical solutions of \eqref{P:intro} can be very irregular, in particular, they may not belong to any Hölder space $C^\alpha(\Omega)$ with $\alpha\in(0,1)$, see \cite[Remark 5.6]{CS22}.
    \item No explicit continuous solution of \eqref{P:intro} with $f\in L^\infty(\Omega)$ and $f>0$ in $\Omega$ is known. In particular, the torsion function, namely the solution of \eqref{P:intro} with $f\equiv 1$, is not known to have a closed formula in any interval.
\end{enumerate}

 These obstacles justify the interest and importance of the analysis of the numerical approximation of solutions to \eqref{P:intro}.  Furthermore, we show that new phenomena appear and that new ideas are required to analyze and implement the finite element method in this setting.

In this work, we focus our analysis to dimension one. This simplified setting allows us to present the key ideas and new concepts in a more transparent way and it has several other advantages; for example, in 1D we can provide an \emph{explicit formula} for the stiffness matrix.  This makes the implementation of the FEM very straightforward and fast in just a few lines of code.  Analyzing problems in 1D has been successfully used for illustrating and identifying important qualitative features of fractional elliptic and parabolic problems, see, for instance, \cite{ABBM18,BdPM18,BHS19,CGH20}. We think that our work provides the community with a compelling and simple-to-use tool to understand the particular pathologies of logarithmic problems and to test and deduce new conjectures in this setting.

Although several of the ideas involved in our 1D analysis can be generalized to dimensions 2 and 3, we believe that there are still important open questions in the 1D case that should be solved before attempting a more ambitious analysis in higher-dimensions.

Our main result (see Theorem~\ref{main:thm:intro} below) states that the FEM numerical approximation of the weak solution converges logarithmically; in particular, we show that the norm of the numerical error can be bounded in terms of $|\ln(h)|^{-\alpha}$, where $h$ is the discretization parameter and $\alpha\in(0,1)$.   This contrasts with the convergence rate $h^\frac{1}{2}|\ln h|$ observed for the fractional Laplacian for $s\in(0,\frac{1}{2})$ (see \cite[Theorem 3.31]{Bor17}) and the well-known rate $h$ for the standard Laplacian.

To state our main result in more detail, let us introduce some notation. Let $L>0$, $\Omega:=(0,L),$ $B_1(\Omega):=(-1,L+1)$. Following \cite{CW19}, consider the Hilbert space
\begin{align*}
\mH(\Omega):=\{u\in L^2(\Omega)\::\: \|u\|_{\mathbb H(\Omega)}<\infty \text{ and }u=0\text{ on }\R\backslash \Omega\},
\end{align*}
where $\|u\|^2_{\mathbb H(\Omega)}:={\mathcal E}(u,u)$ and ${\mathcal E}(\cdot,\cdot)$ is a scalar product on $\mathbb H(\Omega)$ given by
\begin{align*}
    {\mathcal E}(u,v)&=\int_{\R}\int_{B_1(x)}\frac{(u(x)-u(y))(v(x)-v(y))}{|x-y|}\, dy\, dx \qquad \text{ for }u,v\in \mH(\Omega).
\end{align*}
It is known that
\begin{align}\label{c:em}
\mathbb H(\Omega)\hookrightarrow L^2(\Omega)\qquad \text{ is compact,}
\end{align}
see \cite[Theorem 2.1]{CdP18}. Note that the space $\mH(\Omega)$ imposes very little regularity restrictions on its elements, and in particular step functions belong to $\mH(\Omega)$. As a consequence, this space does not allow a notion of trace for its elements.

As mentioned before, the problem \eqref{P:intro} has a variational structure. For $f\in L^2(\Omega)$, we say that $u\in \mH(\Omega)$ is a weak solution of \eqref{P:intro} if $    \cE_{L}(u,v)=\int_{\Omega} f(x) v(x)\, dx$ for all $v\in  \mH(\Omega),$ where
\begin{align}\label{cEL:def}
\cE_L(u,v)=\cE(u,v)+B(u,v) 
\end{align}
and
\begin{align}\label{eq:bilinear_B}
    B(u,v)=-\iint_{|x-y|\geq 1}\frac{u(x)\,v(y)}{|x-y|}dx dy+\rho_1\int_{\R}u v dx.
\end{align}
 The bilinear form $\cE_L$ is \emph{not} coercive in general and it does not induce a norm in $\mH(\Omega)$.

The analysis of FEM in nonlocal problems relies strongly on fine regularity properties of weak solutions. This kind of results face important challenges in the logarithmic setting, where this theory is still underdeveloped; however, recently in \cite{CS22} the regularity of weak solutions in log-Hölder spaces has been studied.  Since these results play a fundamental role in our approach (in particular the estimate \eqref{eq:log_problem} below), it is worth to describe them in detail. Let $\ell:(0,\infty)\to (0,1]$ be the modulus of continuity given by
\begin{align}\label{ell:def}
\ell(\rho) := |\ln(\min(\rho_0,\rho))|^{-1}, \qquad\rho_0:=0.1.
\end{align}
Note that $\ell$ is a non-decreasing concave function with $\ell(0) := \lim_{r\to0^+}\ell(r)=0$. Given $\alpha\in(0,1)$, let
\begin{align}
    \mathcal X^\alpha(\Omega)&:=\left\{u:\mathbb R\to \mathbb R:\|u\|_{\mathcal{X}^\alpha(\Omega)}<+\infty \;\textnormal{ and $u=0$ in }\; \R\setminus \Omega\right\}, \label{Xdef}\\
    \mathcal Y(\Omega)&:=\{f:\Omega\to \mathbb R:\|f\|_{\mathcal Y(\Omega)}<+\infty\},\notag
\end{align}
where
\begin{align*}
    \|u\|_{\mathcal X^\alpha(\Omega)}&:=\sup_{
    \genfrac{}{}{0pt}{}{x,y\in\mathbb R}{x\neq y}
    }\frac{|u(x)-u(y)|}{\ell^{\alpha}(|x-y|)}+\sup_{
\genfrac{}{}{0pt}{}{x,y\in\Omega}{x\neq y}
    }\ell^{1+\alpha}(d(x,y))\frac{|u(x)-u(y)|}{\ell^{1+\alpha}(|x-y|)}, \\
    \|f\|_{\mathcal Y(\Omega)}&:=\|f\|_{L^\infty(\Omega)}+\sup_{
    \genfrac{}{}{0pt}{}{x,y\in\Omega}{x\neq y}
    }\ell^{2}(d(x,y))\frac{|f(x)-f(y)|}{\ell(|x-y  |)}, \\
    d(x)&:=\dist(x,\partial \Omega), \quad d(x,y):=\min(d(x),d(y)).
\end{align*}

The next Fredholm-alternative-type result characterizes the existence and regularity of classical solutions of \eqref{P:intro}.

\begin{theorem}[Theorem 1.1 in \cite{CS22}]\label{eq:regularity}
    There is $\alpha=\alpha(\Omega)\in(0,1)$ such that exactly \underline{one} of the following alternatives holds:
    \begin{enumerate}[label=\roman*)]
        \item For every $f\in\mathcal Y(\Omega)$ there exists a unique classical solution $u\in\mathcal X^{\alpha}(\Omega)$ of
        \begin{equation}\label{eq:log_problem}
            L_{\Delta} u = f \;\textnormal{ in }\; \Omega, \quad u=0 \;\textnormal{ on } \R\setminus \Omega.
        \end{equation}
        Moreover,
        \begin{equation}\label{eq:regularity_classical}
            \|u\|_{\mathcal X^{\alpha}(\Omega)}\leq C\|f\|_{\mathcal Y(\Omega)}
        \end{equation}
        for some constant $C>0$ only depending on $\Omega$.
        \item There is a non-trivial classical solution $u\in {\mathcal X}^{\alpha}(\Omega)$ of $L_{\Delta} u=0$.
    \end{enumerate}
\end{theorem}

As a final ingredient to state our main result, we introduce the notation of the 1D FEM.  Let $N\in\mathbb N$, $h=\frac{L}{N+1}$, $x_i:=ih$, and $\Omega_i:=(x_{i},x_{i+1})$ for $i=0,\ldots,N$. Note that $\overline{\Omega}=\cup_{i=0}^N \overline{\Omega_i}$. Consider the discrete space
\begin{equation*}
    \mathcal V_h=\{v\in \mathcal C^0(\Omega): v=0 \textnormal{ in } \R\setminus \Omega, \;\; v|_{\Omega_i}\in \mathcal P_1, \;\;  i=0,\ldots,N\},
\end{equation*}
where $\mathcal P_1$ is the space of polynomials of degree less than or equal to one.

We say that  $u_h\in \mathcal V_h$ is the \emph{finite element approximation} of the weak solution $u$ of \eqref{P:intro} if \begin{equation}\label{eq:discr_problem:intro}
    \cE_{L}(u_h,v)=\int_{\Omega} f(x) v(x)\, dx \qquad \text{ for all } v\in \mathcal V_h.
\end{equation}

 Our main result is the following.
 \begin{theorem}\label{main:thm:intro}
    Let $L>0$, $\Omega:=(0,L)$, and $f\in \mathcal{Y}(\Omega)$.  Assume that alternative $i)$ of \Cref{eq:regularity} holds for some $\alpha\in(0,1)$. Let $u$ be the (unique weak) solution of \eqref{P:intro} and let $u_h$ be its finite element approximation given by \eqref{eq:discr_problem:intro}. Then there are constants $h_0>0$, $C>0$, and $\alpha\in(0,1)$ such that
    \begin{equation*}
        \|u-u_h\|_{\mathbb H(\Omega)}\leq C\ell^{\alpha}(h)\|f\|_{\mathcal Y}\qquad \text{ for all }h\in(0,h_0),
    \end{equation*}
    where $\ell$ is the modulus of continuity given in \eqref{ell:def}.
 \end{theorem}

 For the proof of Theorem~\ref{main:thm:intro}, we extend the general strategy implemented in \cite{Bor17}, where the FEM is analyzed and used to approximate solutions of the fractional Laplacian.  The overall scheme is as follows: If $u$ is a weak solution that we wish to approximate with finite elements, one defines first an \emph{interpolator} denoted by $I_h u$.  Then, via a localization of the norms (a step needed due to the nonlocal nature of the operator, see Lemma~\ref{lem:localization_enorm}) one can bound the global interpolation error $\|u-I_h u\|_{\mathbb H(\Omega)}$ (see Proposition~\ref{eq:estab_H_c_infty}) in terms of a Hilbert space with logarithmic weights (see \eqref{H:weights:def}).  Then, with a Cea's-Lemma-type result (Lemma~\ref{lem:cea_mg}), we can bound the norm $\|u-u_h\|_{\mathbb H(\Omega)}$ in terms of $\|u-I_h u\|_{\mathbb H(\Omega)}$.  Finally, the regularity estimate \eqref{eq:regularity_classical} comes into play to obtain the final estimate stated in Theorem~\ref{main:thm:intro}.

 However, since the regularity theory for $L_\Delta$ is much less developed than that of $(-\Delta)^s$, key elements in the proofs in \cite{Bor17} have to be changed to fit into the logarithmic Laplacian setting.  For instance, the idea of \emph{trace} plays a key role in order to define the Scott-Zhang interpolator used in \cite{Bor17}. This trace is known to exist in the fractional case because, under suitable assumptions on the right-hand side $f$, regularity estimates for the fractional Laplacian imply that the corresponding weak solutions always belong to $H^{\frac{1}{2}+\eps}(\Omega)$ for some $\eps>0$ (see \cite[Theorem 2.3.6]{Bor17}), which is a Hilbert space with trace.  There is no such result for the logarithmic Laplacian $L_\Delta$.  Although Theorem~\ref{eq:regularity} implies that weak solutions are log-Hölder continuous, it is not known if these solutions belong to a suitable Hilbert space with trace.

 As a consequence, we use an interpolator which does not need the notion of trace (see \eqref{eq:interpolator}) but requires two discontinuous base elements $\varphi_0$ and $\varphi_{N+1}$ at the endpoints of the interval. This interpolator satisfies the constant-preservation property \eqref{Ihc}, localization bounds (Lemma \ref{lem:loc_norm_interp}), and stability in \(L^{\infty}\)-sense (Proposition \ref{eq:estab_H_c_infty}), which are the properties that we need.

 The algorithmic implementation of the FEM can be done, as usual, by computing the entries of the corresponding stiffness matrix numerically. However, here we present a new alternative, which we believe that also gives a new perspective on how to understand the expansion \eqref{intro:exp} from the FEM point of view. Let us be more precise. Typically, in order to obtain the finite element approximation $u_h$ given by \eqref{eq:discr_problem:intro}, a finite element basis $(\varphi_i)_{i=0}^{N+1}$ of $\mathcal V_h$ is considered (see \eqref{eq:def_basis}-\eqref{eq:def_basis_ext}) and a matrix with entries
 \begin{align}\label{s:m:intro}
 \mathcal A_{h}^{L}=(\cE_L(\varphi_i,\varphi_j))_{i,j=0}^N
 \end{align}
 is calculated.  The matrix $\mathcal A_{h}^{L}$ is called the stiffness matrix for the logarithmic Laplacian.  Then, solving the linear algebraic problem $\mathcal A_{h}^{L}\alpha = F,$  where $F=(f_i)_{i=0}^{N+1}$ and $f_i=\int_{\Omega}f(x)\varphi_i(x)\, dx$, one finds a vector $\alpha=(\alpha_i)_{i=0}^{N+1}$ and then $u_h(x):=\sum_{i=0}^{N+1}\alpha_i \varphi_i(x)$  is the finite element approximation and satisfies \eqref{eq:discr_problem:intro}. In this paper, instead of calculating \eqref{s:m:intro} (either directly or numerically), we use the stiffness matrix $\mathcal A_{h}^{s}$ for the fractional Laplacian $(-\Delta)^s$ and show that
 \begin{align}\label{Ahds}
 \mathcal A_{h}^{L}=\partial_s \mathcal A_{h}^{s}|_{s=0},
 \end{align}
 see Lemma~\ref{lem:derivative:s:m}.  This is possible, since a closed formula for the stiffness matrix $\mathcal A_{h}^{s}$ has been obtained in \cite{BH17}\footnote{This reference is an extended preprint version of the published article \cite{BHS19}.}.  Formula \eqref{Ahds} can be seen as a discrete version of the expansion \eqref{intro:exp} in terms of the stiffness matrix.
 
 Because the discontinuous base elements $\varphi_0$ and $\varphi_{N+1}$ were not considered in \cite{BH17}, one needs to calculate only the corresponding entries.  We include a closed formula for $\mathcal A_{h}^{L}$ in Section~\ref{sec:stiffness}.  With this, the code implementation of the FEM for $L_\Delta$ is greatly simplified. We emphasize that the matrix $\mathcal A_{h}^{L}$ only depends on $h$ in the sense that, for a fixed value of $h$, any interval of length $L$ splitted into $N+1=\frac{L}{h}$ subintervals has the same formula for $\mathcal A_{h}^{L}$. As mentioned before, another way to obtain the stiffness matrix $\mathcal A_{h}^{L}$ is to compute the entries \eqref{s:m:intro} either directly or numerically. Both approaches yield the same results.

Using these ideas, we also investigate the (Dirichlet) eigenvalues of the logarithmic Laplacian $L_{\Delta}$ in an interval $(-L,L)$ for some $L>0$. As shown in \cite[Theorem 1.4]{CW19}, the eigenvalues of $L_{\Delta}$ in a bounded domain are a sequence such that $\lambda_1<\lambda_2\leq \ldots\leq\lambda_n\to \infty$ for $n\to\infty$.  Moreover, by using scalings (see \cite[Lemma 2.5]{LW21}) it holds that if $\lambda_k(\Omega)$ denotes the $k$-th eigenvalue of $L_{\Delta}$ in a bounded interval $\Omega$, then $\lambda_k(R\Omega)=\lambda_k(\Omega)-2\ln(R)$ for $R>0$. In Section~\ref{eigenvalue approximation:sec} we first show that the eigenvalues of the stiffness matrix defined in \eqref{s:m:intro} converge to the eigenvalues of the logarithmic Laplacian as the discretization parameter $N$ tends to infinity, see Proposition~\ref{prop:eigen-approx} below. Knowing from the scaling properties that, by sending $L$ to infinity, there are $L_i>0$ such that $\lambda_i=0$ in $(-L_i,L_i)$  for $i\in \N$ (see Figure~\ref{L vs condition}), it is enough to estimate the cases in which $0$ is an eigenvalue of $\mathcal{A}^L_h$. Then, we use scalings to approximate the eigenvalues in a general interval. To be precise, we have the following result.
\begin{proposition}\label{eigenvalue approximation}
For $L>0$ the eigenvalues of $L_{\Delta}$ in $(-L,L)$ are given by $\lambda_k=2\ln(L_k/L)$, $k\in \N$, where, for each $k\in \N$, the value $L_k>0$ is such that zero is an eigenvalue of $L_{\Delta}$ in $(-L_k,L_k)$. Moreover, for $k\in\{1,\ldots,6\}$, $L_k$ is given approximately by
\begin{equation*}
L_1\approx 0.7090, \quad L_2\approx 2.3787,\quad L_3\approx 3.9110,\quad L_4\approx 5.5077, \quad L_5\approx 7.0608,\quad L_6\approx 8.6461.
\end{equation*}
\end{proposition}

In view of Proposition~\ref{eigenvalue approximation}, the first six eigenvalues of $L_{\Delta}$ in $(-1,1)$ are given approximately by
\begin{equation*}
\lambda_1\approx -0.6878, \quad \lambda_2\approx 1.7331,\quad \lambda_3\approx 2.7275,\quad \lambda_4\approx 3.4122, \quad \lambda_5\approx 3.9091,\quad \lambda_6\approx 4.3142.
\end{equation*}
This can be seen in relation to the numerical approximation of the eigenvalues $\lambda_{k}^s$, $k\in\N$, $s\in(0,1)$ of the fractional Laplacian $(-\Delta)^s$ in $(-1,1)$ given in \cite{K12} in combination with the fact that  $\lambda_k^s=1+s\lambda_k+o(s)$ as $s\to 0^+$ for any $k\in\N$, see \cite{FJW22}. Indeed, from \cite[Table 1]{K12} (using only the numerical approximations) we have the following.
\begin{center}
\begin{tabular}{c|cc|cc|cc}
$s$ & $\lambda_1^s$ &$1+s\lambda_1$  & $\lambda_2^s$ &$1+s\lambda_2$ & $\lambda_3^s$ &$1+s\lambda_3$\\
\hline
0.005 & 0.997 & 0.996 & 1.009 & 1.008 & 1.014 & 1.019\\
0.05  & 0.973 & 0.965 & 1.092 & 1.086 & 1.148 & 1.195\\
0.1   & 0.957 & 0.931 & 1.197 & 1.173 & 1.320 & 1.391
\end{tabular}
\end{center}

To close this introduction, we discuss some numerical approximations.  First, we analyze the torsion function, namely, the solution of the equation
\begin{align}\label{t:p:intro}
 L_\Delta \tau = 1\quad \text{ in }(-L,L),\qquad \tau=0\quad \text{ on }\R\backslash (-L,L)
\end{align}
for $L>0$.  In contrast to the case of the fractional Laplacian, the solution of \eqref{t:p:intro} does not have a closed formula. In fact, the numerical approximations show that the shape of the torsion functions changes drastically as the length of the domain $L>0$ increases, see Figure~\ref{Fig:t:f}. In particular, for $L=0.1$ the torsion function is positive in $(-L,L)$.  This is consistent with the known fact that the maximum principle for the logarithmic Laplacian only holds in small domains, see \cite[Corollary 1.9]{CW19}. For $L=1$ we observe that the maximum principle does not hold anymore and the solution is \emph{negative} in $(-L,L)$ (which is sometimes referred to as an antimaximum principle) and for $L=8$ the torsion function exhibits oscillations and changes sign. These last two phenomena have not been described theoretically so far.

\begin{figure}[htb]
	\centering
	\subfloat[$L=0.1$]{
	\includegraphics{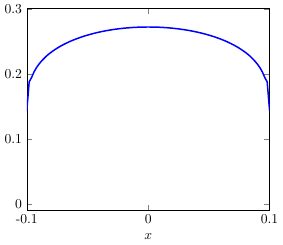}
	}
	\subfloat[$L=1$]{
	\includegraphics{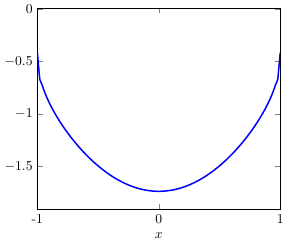}
	} 
	\subfloat[$L=8$]{
	\includegraphics{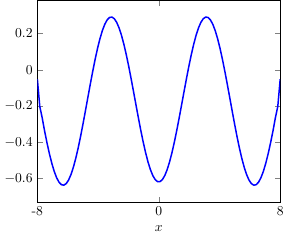}
	}
	\caption{Numerical approximation of the solution to \eqref{t:p:intro} for different values of $L$.}
	\label{Fig:t:f}
\end{figure}

We have also tested our algorithm with some explicit solutions. We refer to Section~\ref{sec:numerics} for more approximations and a discussion on the optimality of the logarithmic rate given in Theorem~\ref{main:thm:intro}.

The paper is organized as follows. In \Cref{sec:three}, we present the localization of the $\mathbb{H}$-norm and some useful approximation properties and results. In Section~\ref{sec:error} we give the proof of Theorem~\ref{main:thm:intro}. We calculate explicitly the stiffness matrix for the logarithmic Laplacian in Section~\ref{sec:stiffness} and we discuss the optimality of the convergence rate in Section~\ref{sec:numerics}. Finally, in Section~\ref{eigenvalue approximation:sec}, we give the proof of Proposition~\ref{eigenvalue approximation} and show that the eigenvalues of the stiffness matrix converge to the eigenvalues of $L_{\Delta}$ as the discretization parameter tends to zero, see Proposition~\ref{prop:eigen-approx}.

\section{Notation and framework} \label{sec:notation}

Let $\rho_1:=2\ln 2 + \psi(\tfrac{1}{2})-\gamma\approx -1.15,$ where $\gamma=-\Gamma'(1)$ is the Euler-Mascheroni constant, and $\psi=\frac{\Gamma'}{\Gamma}$ is the Digamma function.  We use $\|\cdot\|_{L^p(\Omega)}$ to denote the usual $L^p(\Omega)$ norm for $p\in[1,\infty]$.

\subsection{Log-Hölder moduli of continuity}

The modulus $\ell$ (defined in \eqref{ell:def}) satisfies the following property.

\begin{lemma}[Semi-homogeneity]\label{prop1}
There is $c>0$ such that $\ell(\lambda)\leq c\frac{\ell(\lambda r)}{\ell(r)}$ for all $r,\lambda >0.$
\end{lemma}
For a proof, see \cite[Lemma 3.2]{CS22}.

\subsection{Variational setting}
Let $L>0$, $\Omega:=(0,L)$, $B_1(\Omega):=(-1,L+1)$. For $u\in \mH(\Omega)$, we have that
\begin{align*}
    \|u\|^2_{\mathbb H(\Omega)}&=\int_{\R}\int_{B_1(x)}\frac{|u(x)-u(y)|^2}{|x-y|}\, dy\, dx
    =\int_{B_1(\Omega)}\int_{B_1(x)}\frac{|u(x)-u(y)|^2}{|x-y|}\, dy\, dx\\
    &=\int_\Omega\int_\Omega \frac{|u(x)-u(y)|^2}{|x-y|}\, dy\, dx+\int_{\Omega}h_\Omega(x)u(x)^2\, dx,
\end{align*}
where $h_\Omega(x):=\left(\int_{B_1(x)\backslash \Omega}|x-y|^{-1}\, dy - \int_{\Omega\backslash B_1(x)}|x-y|^{-1}\, dy\right).$ For any $\beta>0$,
\begin{align}
    \weH{\beta}&:=\left\{ u\in L^2(\Omega): \int_{\Omega}\int_{\Omega}\frac{|u(x)-u(y)|^2}{|x-y|\ell(|x-y|)^{\beta}}dydx<+\infty \quad\textnormal{ and }\quad  \int_{\Omega}|u(x)|^2\ell^{-\beta}(d(x))dx <+\infty\right\},\label{H:weights:def}
\end{align}
where we recall that $d(x)=\operatorname{dist}(x,\partial \Omega)$ and $\ell$ is given in \eqref{ell:def}.  We endow this space with the norm
\begin{align}\label{Hbetanorm}
 \|u\|_{\mathbb H^\beta(\Omega;\ell)}:=\left(
 \int_{\Omega}\int_{\Omega}\frac{|u(x)-u(y)|^2}{|x-y|\ell(|x-y|)^{\beta}}dydx+\int_{\Omega}|u(x)|^2\ell^{-\beta}(d(x))dx
 \right)^\frac{1}{2}.
\end{align}

\begin{proposition}\label{prop:more_regularity}
Assume that alternative $i)$ of \Cref{eq:regularity} holds for some $\alpha\in(0,1)$. Then, $u\in \mathbb H^{1+\alpha}(\Omega;\ell)$. Moreover, there is $C>0$ only depending on $\Omega$ and $\alpha$ such that $\|u\|_{\mathbb H^{1+\alpha}(\Omega;\ell)}\leq C\|u\|_{\mathcal X^{\alpha}(\Omega)}$.
\end{proposition}
\begin{proof}
Fix $\alpha\in(0,1)$ as given by \Cref{eq:regularity}. Let $x\in \Omega$ and let $y\in\partial \Omega$ be such that $d(x)=|x-y|$. Then $u(y)=0$ and
\begin{align}\label{eq:iden_fro}
    \frac{|u(x)|}{\ell^{\alpha}(d(x))}=\frac{|u(x)-u(y)|}{\ell^{\alpha}(|x-y|)}.
\end{align}
From the definition of the spaces ${\mathbb H}^{1+\alpha}(\Omega;\ell)$ (see \eqref{H:weights:def}), $\mathcal X^\alpha(\Omega)$ (see \eqref{Xdef}), and identity \eqref{eq:iden_fro}, we have that
\begin{align*}\notag
    \|u\|^2_{H^{1+\alpha}(\Omega;\ell)}&=\int_{\Omega}\int_{\Omega}\frac{|u(x)-u(y)|^2}{|x-y|\ell^{1+\alpha}(|x-y|)}\, \d{y} \d{x}+\int_{\Omega}|u(x)|^2\ell^{-1-\alpha}(d(x))\,\d{x} \\
    &\leq \|u\|^2_{\mathcal X^{\alpha}(\Omega)}\left(\int_{\Omega}\int_{\Omega}\frac{\ell^{1+\alpha}(|x-y|)}{|x-y|\ell^{2+2\alpha}(d(x,y))}\,\d{y}\d{x}+ \int_{\Omega}\ell^{-1+\alpha}(d(x))\,\d{x}\right). %\label{a6}
\end{align*}
The claim now follows from \eqref{lem:int:bds}.
\end{proof}

\subsection{The discrete space
\texorpdfstring{$\mathcal V_h$}{Vh}
}

Let $N\in \mathbb N$, $h:=\frac{L}{N+1},$ $x_i:=ih,$ and $\Omega_i:=(x_i,x_{i+1})=(ih,(i+1)h)$ for $i\in\inter{0,N}$, where $\inter{a,b}:=[a,b]\cap\mathbb{N}$ for $a<b$. We consider the discrete space
\begin{equation}\label{Vh:def}
    \mathcal V_h=\{v\in \mathcal C^0(\R): v=0 \textnormal{ in } \R\setminus \Omega, \;\; v|_{\Omega_i}\in \mathcal P_1, \;\;  i\in\inter{0,N}\}
\end{equation}
where $\mathcal P_1$ is the space of polynomials of degree less than or equal to one. 

Let $(\varphi_i)_{i\in\inter{1,N}}$ be the basis of shape functions $\varphi_i$ with compact support in each interval $[x_{i-1},x_{i+1}]$ given by
\begin{align}\label{eq:def_basis}
 \varphi_i(x)=\begin{cases}
        \displaystyle \frac{x-x_{i-1}}{h}, &x\in[x_{i-1},x_i], \\
        \displaystyle \frac{x_{i+1}-x}{h}, &x\in [x_i,x_{i+1}], \\
        0, &\textnormal{elsewhere}.
    \end{cases}
\end{align}
Note that $\varphi_i(x_j)=\delta_{ji}$.  As explained in the introduction, we also consider
\begin{align}\label{eq:def_basis_ext}
    \varphi_0(x)=\begin{cases}
        \displaystyle \frac{x_{1}-x}{h}, &x\in [x_0,x_{1}], \\
        0, &\textnormal{elsewhere},
    \end{cases} \qquad \varphi_{N+1}(x)=\begin{cases}
        \displaystyle \frac{x-x_{N}}{h}, &x\in[x_{N},x_{N+1}], \\
        0, &\textnormal{elsewhere.}
    \end{cases}
\end{align}

For any $R>0$, we also define $B_{R}(\Omega):=\{x\in\R:\dist(x,\Omega)<R\}=(-R,L+R)$.

We also introduce the following notation, for $i\in \inter{0,N}$,
\begin{align}\label{eq:def_Si}
S_{i}:=((i-1)h,(i+2)h)\cap (0,L).
\end{align}
 In particular,
\begin{align*}
    S_i=\bigcup_
    {
    {\substack{j\in\inter{0,N},\\ \overline{\Omega_j}\cap \overline{\Omega_i}\neq \emptyset}}
    }\Omega_j \qquad \text{ for } i\in\inter{0,N}.
\end{align*}
%
% \begin{figure}[h!]
% \begin{center}
%  \includegraphics[width=5cm]{ex.png}
% \end{center}
% \caption{The sets $S_i$.}
% \label{figS}
% \end{figure}
%
%

\section{Preliminary results and properties}\label{sec:three}

\subsection{Localization of the
\texorpdfstring{$\mathbb{H}$}{H}
-norm}\label{sec:localization-norm}

In the following, $C$ denotes possibly different positive constants that are independent of the parameter $h$. 

Let $\Omega=(0,L)$, $h=\frac{L}{N+1}$. Recall that $\Omega_i:=(x_{i},x_{i+1})$ for $i\in\inter{0,N}$ and $\overline{\Omega}=\overline{\cup_{i\in\inter{0,N}}\Omega_i}$. Unlike the usual fractional norm, we note that the $\mathbb H$-norm only has a finite range of interaction, indeed
\begin{align}\label{eq:norm_R_bound}
    \|u\|^2_{\mathbb H(\Omega)} &=\int_{B_1(\Omega)}\int_{B_1(x)}\frac{|u(x)-u(y)|^2}{|x-y|}\, dy\, dx,
\end{align}
because if $\dist(x,\Omega)>1$ and $|x-y|<1$, then $y,x\not\in \Omega$ (\emph{i.e.}, $u(x)=u(y)=0$).  To \emph{localize} this norm we follow the ideas in \cite{Bor17}, which are based on the works \cite{Fae00,Fae02}.

\begin{lemma}\label{lem:localization_enorm}
Let $L>0$ be fixed and $N\in\mathbb N$ be large enough so that $h\in(0,\frac{1}{2})$. Then, there exists $C>0$ only depending on $L$ such that
   \begin{align*}\notag
    \|u\|^2_{\mathbb H(\Omega)} &\leq \sum_{i\in\inter{0,N}}\int_{\Omega_i}\int_{S_i}\frac{|u(x)-u(y)|^2}{|x-y|}dy\,dx+C\ell^{-1}(h)\sum_{i\in\inter{0,N}}\int_{\Omega_i}|u(x)|^2\,dx +C\sum_{i\in\{0,N\}}\int_{\Omega_i}|u(x)|^2\ell^{-1}(d(x))\,dx
\end{align*}
for any $u\in \mathbb{H}(\Omega)$ and $\ell$ as in \eqref{ell:def}.
\end{lemma}
\begin{proof}
We note that for $x\in B_1(\Omega)$, $B_1(x)\subset B_2(\Omega)$, therefore
\begin{align}\notag
    \|u\|_{\mathbb H(\Omega)}^2&\leq \int_{B_2(\Omega)}\int_{B_2(\Omega)}\frac{|u(x)-u(y)|^2}{|x-y|}dy\, dx\\ \notag
    &=\int_{\Omega}\int_{\Omega}\frac{|u(x)-u(y)|^2}{|x-y|}dy\,dx+2\int_{\Omega}|u(x)|^2\int_{B_2(\Omega)\setminus B_h(\Omega)}\frac{1}{|x-y|}dy\,dx+2\int_{\Omega}|u(x)|^2\int_{B_h(\Omega)\setminus \Omega}\frac{1}{|x-y|}dy\,dx\\ &=: I_1+2I_2+2I_3,\label{eq:part_init}
\end{align}
where we have used that $u=0$ in $\R\setminus \Omega$. 

We proceed to estimate each $I_i$. For $I_1$, we see that  
\begin{equation*}
    I_1=\sum_{i\in\inter{0,N}}\left(\int_{\Omega_i}\int_{S_i}\frac{|u(x)-u(y)|^2}{|x-y|}\,dy\,dx+\int_{\Omega_i}\int_{D_i}\frac{|u(x)-u(y)|^2}{|x-y|}\,dy\,dx\right),
\end{equation*}
where $D_i=\Omega\setminus S_i$. We keep the first integrals as they are. For the second ones, we have, by triangle inequality,
\begin{align*}\notag
    \int_{\Omega_i}&\int_{D_i} \frac{|u(x)-u(y)|^2}{|x-y|}dy\,dx\leq 2\int_{\Omega_i}|u(x)|^2\int_{D_i}|x-y|^{-1}dy\,dx+2\int_{D_i}|u(y)|^2\int_{\Omega_i}|x-y|^{-1}dx\,dy=:2J_{i,1}+2J_{i,2}.
\end{align*}
Following \cite[Prop. 1.2.24]{Bor17}, it is not difficult to see that $\sum_{i}{J_{i,1}}=\sum_{i}{J_{i,2}}$. Hence,
\begin{align*}
    I_1\leq \sum_{i\in\inter{0,N}}\left(\int_{\Omega_i}\int_{S_i}\frac{|u(x)-u(y)|^2}{|x-y|}dy\,dx+4 J_{i,1}\right).
\end{align*}
To bound $J_{i,1}$, note that
\begin{align}\notag
    J_{i,1}&\leq \int_{\Omega_i}|u(x)|^2\int_{B_{L}(x)\setminus B_{\frac{h}{2}}(x)}|x-y|^{-1}dy\,dx \\\label{eq:est_Ij1}
    &\leq 2\int_{\Omega_i}|u(x)|^2\int_{\frac{h}{2}}^{L}\rho^{-1}d\rho\,dx
    =2\left(-\log\left(\tfrac{h}{2L}\right)\right)\int_{\Omega_i}|u(x)|^2\,dx
    \leq C\ell^{-1} (h)\int_{\Omega_i}|u(x)|^2\,dx
\end{align}
for some $C>0$ independent of $h$ depending on $\Omega$ and where we have used that $h\in(0,1/2)$ in the last inequality. Thus, we have proved that
\begin{equation}\label{eq:est_I1_final}
    I_1\leq \sum_{i\in\inter{0,N}}\left(\int_{\Omega_i}\int_{S_i}\frac{|u(x)-u(y)|^2}{|x-y|}dy\,dx + C\ell^{-1}(h)\int_{\Omega_i}|u(x)|^2\,dx\right).
\end{equation}

For $I_2$, note that $I_2=\sum_{i\in\inter{0,N}}\int_{\Omega_i}|u(x)|^2\int_{B_2(\Omega)\setminus B_{h}(\Omega)}|x-y|^{-1}dy\,dx.$ Since $\dist(\Omega_i,B_2(\Omega)\setminus B_h(\Omega))>h$ for all $i\in\inter{0,N}$, we can argue as we did for \eqref{eq:est_Ij1} to obtain
\begin{align}\label{eq:est_I2}
    I_2\leq C^\prime\ell^{-1}(h)\sum_{i\in\inter{0,N}}\int_{\Omega_i}|u(x)|^2\,dx
\end{align}
for some $C^\prime>0$ only depending on $\Omega$. To estimate $I_3$, we begin by writing
\begin{align}\label{eq:est_I3_init}
    I_3=H_0+H_N+\sum_{i\in\inter{1,N-1}}H_i,
\end{align}
where $H_i=\int_{\Omega_i}|u(x)|^2\int_{B_h(\Omega)\setminus \Omega}|x-y|^{-1}dy\,dx$ for $i\in\inter{0,N}$. Note that $\dist(\Omega_i,B_h(\Omega)\setminus\Omega)>h$ for $i\in\inter{1,N-1}$, so the second sum of \eqref{eq:est_I3_init} can be estimated as we did for $I_2$. On the other hand, for indices $i\in\{0,N\}$, we have that for some $R>0$ large enough
\begin{align*}
    H_i&\leq \int_{\Omega_i}|u(x)|^2\int_{B_R(0)\setminus B_{d(x)}(x)}|x-y|^{-1}dy\,dx \\
    &\leq \int_{\Omega_i}|u(x)|^2\left(-\log\left(\tfrac{d(x)}{R}\right)\right)\,dx \leq C \int_{\Omega_i}|u(x)|^2\ell^{-1}(d(x))\,dx,
\end{align*}
where $C>0$ only depends on $\Omega$. Therefore,
\begin{align}\label{eq:est_I3}
    I_3\leq C \sum_{i\in\{0,N\}}\int_{\Omega_i}|u(x)|^2\ell^{-1}(d(x))\,dx + C\ell^{-1}(h)\sum_{i\in\inter{1,N-1}}\int_{\Omega_i}|u(x)|^2\,dx.
\end{align}
Finally, putting together \eqref{eq:est_I1_final}, \eqref{eq:est_I2}, and \eqref{eq:est_I3} gives the desired result. 
\end{proof}

\subsection{A density argument}

Recall the shape functions given in \eqref{eq:def_basis} and the space $\mathcal V_h$ defined in \eqref{Vh:def}. We define the (quasi-) interpolator $I_h:L^2(\Omega)\to \mathcal V_h$ by
\begin{equation}\label{eq:interpolator}
I_h v(x)=\sum_{k\in\inter{0,N+1}}a_k \varphi_k(x),\qquad \text{ where }
a_k:=\frac{\int_{\Omega}v(x)\varphi_k(x)\,dx}{\int_{\Omega}\varphi_k(x)\,dx}.
\end{equation}
This allows to have the important property that
\begin{align}\label{Ihc}
I_h 1 = 1\qquad \text{ in $\Omega$,}
\end{align}
namely, that the interpolation of a constant function in $\Omega$ is the constant itself.

\begin{lemma}\label{lem:bounds_aks}
Let $v\in L^2(\Omega)$ and $a_k$ as in \eqref{eq:interpolator}. There is $C>0$ independent of $h$ and $v$ such that
\begin{equation}\label{eq:est_ak_H}
    |a_k|\leq C h^{-1/2} \|v\|_{L^2(S_k)}\qquad \text{ for all } k\in\inter{0,N+1}
\end{equation}
with $S_k$ as defined in \eqref{eq:def_Si}. If, in addition, $v\in L^\infty(\Omega)$, we have
\begin{equation}\label{eq:est_ak_Linf}
    |a_k|\leq C \|v\|_{L^\infty(\Omega)} \quad\text{ for all } k\in\inter{0,N+1}.
\end{equation}
\end{lemma}
\begin{proof}
Let $k\in\inter{0,N+1}$ and $v\in L^2(\Omega)$ be an arbitrary given function. We have that $|a_k|\leq \frac{C}{h}\int_{\Omega}|v||\varphi_k|dx.$ Note that $\textnormal{supp}\,\varphi_k\subset S_k$ for each $k\in\inter{0,N+1}$. Using H\"older's inequality, $|a_k|\leq C h^{-1}\|v\|_{L^1(S_k)}\leq Ch^{-1/2}\|v\|_{L^2(S_k)}.$ Similarly,   $|a_k|\leq \frac{C}{h}\|v\|_{L^\infty(\Omega)}\int_{S_k}|\varphi_k| dx \leq C\|v\|_{L^\infty(\Omega)}.$
\end{proof}

\begin{lemma}\label{lem:poincare_type}
Let $U\subset \R$ be an open bounded interval, $v\in C^1(\overline{U})$, and $\overline v=\frac{1}{|U|}\int_{U}v$; then,
\begin{align*}
\|v-\overline{v}\|_{L^2(U)}^2\leq \|v'\|^2_{L^\infty(U)}|U|^3.
\end{align*}
\end{lemma}
\begin{proof}
By definition and using Jensen's inequality,
\begin{align*}
    \int_{U}|v(x)-\overline{v}|^2dx=\frac{1}{|U|^2}\int_{U}\left|\int_{U}(v(x)-v(y))dy\right|^2dx \leq \frac{1}{|U|}\int_{U}\int_{U}|v(x)-v(y)|^2dy dx,
\end{align*}
whence $\int_{U}|v(x)-\overline{v}|^2dx \leq \int_{U}\int_{U}\frac{|v(x)-v(y)|^2}{|x-y|}dy dx.$ By mean value theorem, since $|x-y|\leq |U|$ for all $x,y\in U$,
\begin{equation*}
    \int_{U}|v(x)-\overline{v}|^2dx 
    \leq \|v'\|^2_{L^\infty(U)}\int_{U}\int_{U}|x-y|dy dx
    \leq \|v'\|^2_{L^\infty(U)}|U|^3.\qedhere
\end{equation*}
\end{proof}

\begin{lemma}\label{lem:inter_l_infty}
    Let $v\in L^\infty(\Omega)$. Then, $\|I_h v\|_{L^\infty(\Omega)}\leq C\|v\|_{L^\infty(\Omega)}$ for some $C>0$ uniform with respect to $h$.
\end{lemma}
\begin{proof}
Let $x\in \Omega$ be arbitrary. From \eqref{eq:interpolator} and \eqref{eq:est_ak_Linf},
\begin{equation*}
|I_h v(x)|
=\left|\sum_{k\in\inter{0,N+1}}a_k \varphi_k(x)\right|
\leq C\|v\|_{L^\infty(\Omega)}\sum_{k\in\inter{0,N+1}}\varphi_k(x)=C\|v\|_{L^\infty(\Omega)}. \qedhere
\end{equation*}
\end{proof}

\begin{lemma}\label{lem:loc_norm_interp}
Let $i\in\inter{0,N}$, $v\in \mathbb H(\Omega)$, and $a_k$ as in \eqref{eq:interpolator}.  There is $C>0$ independent of $h$ such that
\begin{enumerate}[label=\upshape(\roman*)]
\item $\displaystyle \int_{\Omega_i}\int_{S_i}\frac{|I_hv(x)-I_hv(y)|^2}{|x-y|}dy dx\leq Ch\sum_{k\in I_i} a_k^2$,
\item  $\displaystyle \|I_h u\|^2_{L^2(\Omega_i)}\leq Ch\sum_{k\in I_i^\prime} a_k^2$,
\end{enumerate}
where
\begin{gather}\label{def_Ii}
I_i:=\{j\in\inter{0,N+1}:\operatorname{supp}(\varphi_j)\cap S_i\neq 0\}=\{i-1,i,i+1,i+2\}\cap \inter{0,N+1}, \\ \label{def_Iiprime}
I^\prime_i:=\{j\in\inter{0,N+1}:\varphi_j(x)\neq 0 \ \textnormal{for $x\in \Omega_i$}\}=\{i,i+1\}.
\end{gather}
\end{lemma}
\begin{proof}
We begin by proving item (i). Given $i\in\inter{0,N}$, let $x\in \Omega_i$ and $y\in S_i$ such that $x\neq y$.  From definitions \eqref{eq:def_basis}, \eqref{eq:interpolator}, and \eqref{def_Ii} it can be readily seen that
\begin{equation}\label{a1}
    I_hv(x)-I_hv(y)=\sum_{k\in I_i} a_k(\varphi_k(x)-\varphi_k(y)).
\end{equation}
Since $I_i$ has at most 4 indices, squaring both sides of \eqref{a1} and using triangle inequality,
\begin{equation}\label{eq:est_interp_1}
    \int_{\Omega_i}\int_{S_i}\frac{|I_hv(x)- I_hv(y)|^2}{|x-y|} dy dx \leq C \sum_{k\in I_i} a_k^2\int_{\Omega_i}\int_{S_i} \frac{|\varphi_k(x)- \varphi_k(y)|^2}{|x-y|} dy dx.
\end{equation}

Noting that the functions $\varphi_k$ are Lipschitz continuous and verify that $|\varphi_k^\prime|\leq 1/h$ for all $k\in\inter{0,N+1}$, we can use the mean value theorem to deduce
\begin{equation}\label{eq:est_interp_2}
    \int_{\Omega_i}\int_{S_i}\frac{|\varphi_k(x)-\varphi_k(y)|^2}{|x-y|}dy dx \leq \frac{1}{h^2}\int_{\Omega_i}\int_{S_i}|x-y|dy dx \leq Ch,
\end{equation}
since $|x-y|\leq 2h$ for all $x\in \Omega_i$ and $y\in S_i$. Combining \eqref{eq:est_interp_1} and \eqref{eq:est_interp_2} yields (i). 

For (ii), note that from \eqref{eq:def_basis}, \eqref{eq:interpolator}, and \eqref{def_Iiprime}, we have
\begin{align*}
\|I_h u\|^2_{L^2(\Omega_i)}
\leq \sum_{k\in I_i^\prime}|a_k|^2\int_{\Omega_i}|\varphi_k(x)|^2\, dx
\leq C\sum_{k\in I_i^\prime}|a_k|^2|\Omega_i|
=Ch\sum_{k\in I_i^\prime}|a_k|^2.
\end{align*}
\end{proof}

\begin{lemma}\label{lem:stab_dif_interp}
Let $i\in\inter{0,N}$ and $v\in C_c^\infty(\Omega)$. There is $C>0$ independent of $h$ and $v$ such that
\begin{align}\label{eq:approxim_estimate}
    \int_{\Omega_i}\int_{S_i}\frac{|(v-I_h v)(x)-(v-I_h v)(y)|^2}{|x-y|}dy dx \leq C \|v^\prime\|^2_{L^\infty(\Omega)} h^3.
\end{align}
\end{lemma}
\begin{proof}
Let $i\in\inter{0,N}$, $v\in C_c^\infty(\Omega)$, and $\overline {v}:=\frac{1}{5h}\int_{x_i-2h}^{x_i+3h}v\,dx$.
By \eqref{Ihc}, $v-I_h v=v-\overline{v}+I_h(\overline v - v)$, thus
\begin{align}\notag
    \int_{\Omega_i}&\int_{S_i}\frac{|(v-I_h v)(x)-(v-I_h v)(y)|^2}{|x-y|} dy dx \\
    &\leq 2\int_{\Omega_i}\int_{S_i}\frac{|(v-\overline{v})(x)-(v-\overline{v})(y)|^2}{|x-y|}dy dx + 2 \int_{\Omega_i}\int_{S_i}\frac{|I_h(\overline v-{v})(x)-I_h(\overline v-{v})(y)|^2}{|x-y|}dy dx\notag
    \\& =:2J_1+2J_2.\label{eq:est_dif_inter_1}
\end{align}

Using that $\overline v$ is a constant function, we can use mean value theorem to bound $J_1$ as
\begin{equation}\label{eq:est_J1}
    J_1 = \int_{\Omega_i}\int_{S_i}\frac{|v(x)-v(y)|^2}{|x-y|}dy dx \leq C |v^\prime|^2_{\infty}\int_{S_i}\int_{S_i}|x-y|dy dx \leq C{|v^\prime|^2_{\infty}}h^3,
\end{equation}
since $|x-y|\leq 3h$ for $x,y\in S_i$.

On the other hand, we can use item (i) of \Cref{lem:loc_norm_interp} for $J_2$ and obtain
\begin{equation}\label{eq:est_J2}
    J_2 \leq Ch\sum_{k\in I_i} b_k^2,
\end{equation}
where $b_k:=\frac{1}{h}\int_{\Omega}(v-\overline v)\varphi_k dx$ and $I_i$ as defined in \eqref{def_Ii}. Since $|\varphi_k|\leq 1$,
\begin{equation}\label{abk}
    |b_k|^2
    \leq \frac{1}{h^2}\left(\int_{S_k} |v-\overline v|dx\right)^2
    \leq \frac{C}{h}\int_{S_k}|v-\overline{v}|^2\, dx
    \leq \frac{C}{h}\int_{x_i-2h}^{x_i+3h}|v-\overline{v}|^2\, dx
\end{equation}
for $k\in I_i.$ Using \Cref{lem:poincare_type}, we can estimate the right-hand side of the above expression and obtain
\begin{equation}\label{eq:est_bks}
    |b_k|^2\leq C{\|v^\prime\|^2_{L^\infty(\Omega)}}h^2.
\end{equation}
Collecting estimates \eqref{eq:est_dif_inter_1}--\eqref{eq:est_bks} gives the desired result.
\end{proof}

\begin{remark}\label{rem:l2_stability}
By making minor adjustments to the proof of  \Cref{lem:stab_dif_interp}, we can use item (ii) of \Cref{lem:loc_norm_interp} to obtain the following $L^2$-stability estimate: for $i\in\inter{0,N}$, $\|v-I_hv\|^2_{L^2(\Omega_i)} \leq C{\|v^\prime\|^2_{L^\infty(\Omega)}} h^3$ with $v\in C_c^\infty(\Omega).$ Indeed, let $i\in\inter{0,N}$, $U=({x_i-2h},{x_{i}+2h})$, and $\overline{v}:=\frac{1}{U}\int_U v\, dx$; then, by Lemmas~\ref{lem:poincare_type} and~\ref{lem:loc_norm_interp},
\begin{align}
\|v-I_hv\|^2_{L^2(\Omega_i)}
\leq C\|v-\overline{v}\|_{L^2(\Omega_i)}^2 + \|I_h(v-\overline{v})\|^2_{L^2(\Omega_i)}
\leq C{\|v^\prime\|^2_{L^\infty(\Omega)}}h^3+Ch\sum_{k\in I_i^\prime}a_k^2,\label{a2}
\end{align}
where $a_k=\frac{1}{h}\int_{x_{k-1}}^{x_{k+1}} (v-\overline{v})\varphi_k\, dx$ and $|a_k|\leq \frac{1}{h}\int_{x_i-2h}^{x_{i}+2h}|v-\overline{v}|\, dx$. By Lemma~\ref{lem:poincare_type}, $
    \|v-\overline{v}\|_{L^2(U)}^2\leq C{\|v^\prime\|^2_{L^\infty(\Omega)}}h^3.$ The claim now follows from \eqref{a2}.
\end{remark}

\begin{proposition}\label{eq:estab_H_c_infty}
For any $h\in(0,1/2)$ and any $v\in C_c^\infty(\Omega)$,
\begin{equation}\label{eq:rate_H_c_infty}
    \|v-I_h v\|^2_{\mathbb H(\Omega)}\leq C(\|v\|_{L^\infty(\Omega)},\|v'\|_{L^\infty(\Omega)},L)h\,\ell^{-1}(h),
\end{equation}
where $C(\|v\|_{L^\infty(\Omega)},\|v'\|_{L^\infty(\Omega)},L)>0$ is a constant that only depends on $L$, $\|v\|_{L^\infty(\Omega)}$, $\|v'\|_{L^\infty(\Omega)}$.
\end{proposition}

\begin{proof}
Along the proof, $C>0$ denotes possibly different constants that depend at most on $L$, $\|v\|_{L^\infty(\Omega)}$, and $\|v'\|_{L^\infty(\Omega)}$.  Combining \eqref{eq:approxim_estimate} and \Cref{rem:l2_stability} with \Cref{lem:localization_enorm} we obtain
\begin{align*}
    \|v-I_h v\|^2_{\mathbb H(\Omega)}&\leq  C h^3(N+1)+Ch^3(N+1)\ell^{-1}(h)+C\sum_{i\in\{0,N\}}\int_{\Omega_i}|v-I_h v|^2\ell^{-1}(d(x))\,dx \\
    &\leq Ch^2+Ch^2\ell^{-1}(h)+C\sum_{i\in\{0,N\}}\int_{\Omega_i}|v-I_h v|^2\ell^{-1}(d(x))\,dx,
\end{align*}
because $h=\frac{L}{N+1}$. Since $v\in L^\infty(\Omega)$, Lemmas \ref{lem:inter_l_infty} and \ref{lem:ellbeta} yield that
\begin{align*}
    \|v-I_h v\|^2_{\mathbb H(\Omega)} &\leq Ch^2+Ch^2\ell^{-1}(h)+C\left(\|v\|_{L^\infty(\Omega)}^2+\|I_hv\|^2_{L^\infty(\Omega)}\right)\sum_{i\in\{0,N\}}\int_{\Omega_i}\ell^{-1}(d(x))\,dx\leq C h\,\ell^{-1}(h),
\end{align*}
as claimed.
\end{proof}

We are in position to prove the main result of this section.

\begin{theorem}\label{density:thm}
    For any $v\in\mathbb H(\Omega)$ there exists $v_h\in\mathcal V_h$ such that $\|v-v_h\|_{\mathbb H(\Omega)}\to 0$ as $h\to 0$. 
\end{theorem}

\begin{proof}
From the definition of $\mathbb H(\Omega)$, it is clear that Lipschitz continuous functions which are zero on $\R\setminus \Omega$ belong to $\mathbb H(\Omega)$, hence $\mathcal V_h$ is a subspace of $\mathbb H(\Omega)$.

By Hilbert projection theorem, for every $v\in\mathbb H(\Omega)$ there exists a unique element $v_h\in\mathcal \mathcal V_h$ such that
\begin{equation}\label{eq:prop_infimo}
    \|v-v_h\|_{\mathbb H(\Omega)}=\inf_{w\in\mathcal V_h}\|v-w\|_{\mathbb H(\Omega)}. 
\end{equation}
Actually, $v_h=\Pi_h v$ where $\Pi_h$ is the orthogonal projection from $\mathbb H(\Omega)$ onto $\mathcal V_h$. Additionally, $\|\Pi_h\|_*=1$ where $\|\Pi_h\|_*:=\sup\{\|\Pi_hv\|_{\mathbb H(\Omega)}: v\in\mathbb{H},\;\; \|v\|_{\mathbb H(\Omega)}\leq 1\}$.

Let $\epsilon>0$. We recall that $C_c^\infty(\Omega)$ is dense in $\mathbb H(\Omega)$. Thus, for every $v\in\mathbb H(\Omega)$ there exists $w\in C_c^\infty(\Omega)$ such that  $
    \|v-w\|_{\mathbb H(\Omega)}<\epsilon/3.$ By triangle inequality, we have
\begin{align*}
    \|v-v_h\|_{\mathbb H(\Omega)}&\leq \|v-w\|_{\mathbb H(\Omega)}+\|w-\Pi_h w\|_{\mathbb H(\Omega)}+\|\Pi_h w-\Pi_h v\|_{\mathbb H(\Omega)} \\
    &\leq \frac{2\epsilon}{3}+\|w-I_hw\|_{\mathbb H(\Omega)} \leq \frac{2\epsilon}{3}+{C(\|w\|_{L^\infty(\Omega)},\|w'\|_{L^\infty(\Omega)},L)}h\, \ell^{-1}(h),
\end{align*}
where we have used that $\|\Pi_h\|_*=1$, property \eqref{eq:prop_infimo}, and \Cref{eq:estab_H_c_infty}. {Finally, we choose $h>0$ small enough so that ${C(\|w\|_{L^\infty(\Omega)},\|w'\|_{L^\infty(\Omega)},L)}h\, \ell^{-1}(h)\leq \frac{\epsilon}{3}$ (note that $w\in C^\infty_c(\Omega)$ is independent of $h$)}. This yields the claim.
\end{proof}

\subsection{Quasi-interpolation estimates}\label{sec:estimates}

\begin{lemma}\label{lem:poincare_type_weight}
Let $\alpha>0$, $U\subset \R$ be an open bounded interval, $v\in \mathbb H^{\alpha}(U;\ell)$, and $\overline v=\frac{1}{|U|}\int_{U}v\, dx$, then,
\begin{equation*}
    \|v-\overline{v}\|_{L^2(U)}^2\leq \ell^{\alpha}(|U|)\int_{U}\int_{U}\frac{|v(x)-v(y)|^2}{|x-y|\ell^{\alpha}(|x-y|)}dy dx.
\end{equation*}
\end{lemma}
\begin{proof}
By Jensen's inequality,
\begin{align*}
    \int_{U}|v(x)-\overline{v}|^2dx
    &=\frac{1}{|U|^2}\int_{U}\left|\int_{U}(v(x)-v(y))dy\right|^2dx \leq \frac{1}{|U|}\int_{U}\int_{U}|v(x)-v(y)|^2dy dx\\
&\leq \int_{U}\int_{U}\frac{|v(x)-v(y)|^2}{|x-y|}dy dx \leq \ell^{\alpha}(|U|)\int_{U}\int_{U}\frac{|v(x)-v(y)|^2}{|x-y|\ell^{\alpha}(|x-y|)}dy dx.
\end{align*}
This ends the proof. 
\end{proof}

We now introduce some additional notation for the localization of the norms. Let
\begin{align*}
 T_i=((i-2)h,(i+3)h)\cap(0,L).
 \end{align*}
Notice that $T_i=\cup_{i\in I_i} \operatorname{supp}(\varphi_i)$.

\begin{lemma}\label{lem:stab_dif_interp_weight}
Let $\alpha>0$, $i\in\inter{0,N}$, $v\in \weH{\alpha}$, then
\begin{align}
    \int_{\Omega_i}\int_{S_i}\frac{|(v-I_h v)(x)-(v-I_h v)(y)|^2}{|x-y|}dy dx &\leq C  \ell^{\alpha}(h)\int_{T_i}\int_{T_{i}}\frac{|v(x)-v(y)|^2}{|x-y|\ell^{\alpha}(|x-y|)}dy dx, \label{eq:approxim_estimate_weight}\\
    \|v-I_h v\|_{L^2(\Omega_i)}^2&\leq C\ell^{\alpha}(h)\int_{T_{i}}\int_{T_{i}}\frac{|v(x)-v(y)|^2}{|x-y|\ell^{\alpha}(|x-y|)}dy dx.\label{a4}
\end{align}
\end{lemma}
\begin{proof}
 The proof is analogous to the one of \Cref{lem:stab_dif_interp}, where $J_1$ is estimated as
\begin{equation}\label{eq:est_J1_weight}
    J_1 = \int_{\Omega_i}\int_{S_i}\frac{|v(x)-v(y)|^2}{|x-y|}dy dx \leq C\ell^{\alpha}(h)\int_{T_i}\int_{T_i}\frac{|v(x)-v(y)|^2}{|x-y|\ell^{\alpha}(|x-y|)}dy dx
\end{equation}
and $J_2$ is estimated with \Cref{lem:poincare_type_weight}.
\end{proof}

Recall the definition of $\|\cdot\|_{\mathbb H^{\beta}(\Omega;\ell)}$ given in \eqref{Hbetanorm}.
\begin{lemma}\label{new:lem}
For every $\beta>0$, there is $C=C(\beta)>0$ such that
\begin{align}\label{bdr:est}
 \int_{0}^h |I_h u(x)|^2 \ell^{-\beta}(x)\, dx\leq C\|u\|^2_{\mathbb H^{\beta}(\Omega;\ell)}.
\end{align}
In particular,  $\sum_{i\in\{0,N\}}\int_{\Omega_i}|(u-I_h u)(x)|^2\ell^{-1}(d(x))\,dx\leq C\ell^a(h)\|u\|^2_{\mathbb H^{1+a}(\Omega;\ell)}$ for every $a>0$.
\end{lemma}
\begin{proof}
In this proof, $C$ denotes possibly different positive constants independent of $h$ and $u$. Note that
\begin{align*}
 \int_{0}^h |I_h u(x)|^2 \ell^{-\beta}(x)\, dx
 \leq 2\int_{0}^h (a_0^2\varphi_0^2+a_1^2\varphi_1^2) \ell^{-\beta}(x)\, dx,
\end{align*}
where, by Cauchy-Schwarz inequality,
\begin{align*}
 a_0^2
 =\frac{C}{h^2}\left(\int_0^h u\varphi_0\, dx \right)^2
 \leq C\frac{\|\varphi_0\|^2_{L^2(\Omega_0)}}{h^2}\|u\|^2_{L^2(\Omega_0)}
 \leq C\frac{1}{h}\ell^\beta(h)\int_0^h|u(x)|^2\ell^{-\beta}(x)\, dx
 \leq C\frac{\ell^\beta(h)}{h}\|u\|^2_{\mathbb H^{\beta}(\Omega;\ell)}.
\end{align*}
Similarly, $a_1^2=\frac{C}{h^2}\left(\int_0^{2h} u\varphi_1\, dx \right)^2\leq C\frac{\ell^\beta(h)}{h}\|u\|^2_{\mathbb H^{\beta}(\Omega;\ell)}$. Observe that, by Lemma \ref{lem:ellbeta},
\begin{align*}
 \int_{0}^h a_0^2\varphi_0^2\ell^{-\beta}(x)\, dx
 \leq a_0^2\int_{0}^h \ell^{-\beta}(x)\, dx
 \leq C\frac{\ell^\beta(h)}{h}\|u\|^2_{\mathbb H^{\beta}(\Omega;\ell)} h \ell^{-\beta}(h)=C\|u\|^2_{\mathbb H^{\beta}(\Omega;\ell)}.
\end{align*}
Analogously, $\int_{0}^{h} a_1^2\varphi_1^2\ell^{-\beta}(x)\, dx \leq C\|u\|^2_{\mathbb H^{\beta}(\Omega;\ell)},$ and \eqref{bdr:est} follows. By the definition of $\|u\|^2_{\mathbb H^{1}(\Omega;\ell)}$, \eqref{bdr:est}, the fact that
$\ell^{\alpha}(d(x))\leq \ell^{\alpha}(h)$ for $x\in\Omega_i$ with $i\in\{0,N\}$, and arguing by  symmetry, we easily deduce that
\begin{align*}
\int_{\Omega_i}|(u-I_h u)(x)|^2\ell^{-1}(d(x))\,dx
 \leq \ell^a(h)\int_{\Omega_i}|(u-I_h u)(x)|^2\ell^{-1-a}(d(x))\,dx
 \leq C \ell^a(h)\|u\|^2_{\mathbb H^{1+a}(\Omega;\ell)}
\end{align*}
for $i\in\{0,N\}$, as claimed.
\end{proof}

\begin{proposition}\label{prop:est_interpolator}
Let $\alpha>0$ and $u\in\weH{1+\alpha}$, then
\begin{align}\label{eq:interp_final}
    \|u-I_hu\|^2_{\mathbb H(\Omega)}&\leq C\ell^{\alpha}(h) \|u\|^2_{\mathbb H^{1+\alpha}(\Omega;\ell)},
\end{align}
\end{proposition}
\begin{proof}
Using \Cref{lem:localization_enorm}, we have that
\begin{align*}
    \|u-I_h u\|^2_{\mathbb H(\Omega)}
    &\leq \sum_{i\in\inter{0,N}}\int_{\Omega_i}\int_{S_i}\frac{|(u-I_h u)(x)-(u-I_h u)(y)|^2}{|x-y|}dy\,dx\\
    &+C\ell^{-1}(h)\sum_{i\in\inter{0,N}}\int_{\Omega_i}|(u-I_h u)(x)|^2\,dx +C\sum_{i\in\{0,N\}}\int_{\Omega_i}|(u-I_h u)(x)|^2\ell^{-1}(d(x))\,dx.
\end{align*}
Using estimates \eqref{eq:approxim_estimate_weight}, \eqref{a4}, and Lemma \ref{new:lem},
\begin{align*}
 \sum_{i\in\inter{0,N}}\int_{\Omega_i}\int_{S_i}\frac{|(u-I_h u)(x)-(u-I_h u)(y)|^2}{|x-y|}dy\,dx &\leq
 \ell^{1+\alpha}(h)\sum_{i\in\inter{0,N}}\int_{T_i}\int_{T_i}\frac{|u(x)-u(y)|^2}{|x-y|\ell^{1+\alpha}(|x-y|)}dy dx,
 \\
\ell^{-1}(h)\sum_{i\in\inter{0,N}}\int_{\Omega_i}|(u-I_h u)(x)|^2\,dx &\leq \ell^{a}(h)\sum_{i\in\inter{0,N}}\int_{T_i}\int_{T_i}\frac{|u(x)-u(y)|^2}{|x-y|\ell^{1+\alpha}(|x-y|)}dy dx,\\
 \sum_{i\in\{0,N\}}\int_{\Omega_i}|(u-I_h u)(x)|^2\ell^{-1}(d(x))\,dx
 &\leq C \ell^a(h)\|u\|^2_{\mathbb H^{1+a}(\Omega;\ell)}.
\end{align*}
Thus,
\begin{align}
    \|u-I_h u\|^2_{\mathbb H(\Omega)}&\leq
    C \ell^{\alpha}(h)\sum_{i\in\inter{0,N}}\int_{T_i}\int_{T_i}\frac{|u(x)-u(y)|^2}{|x-y|\ell^{1+\alpha}(|x-y|)}dy dx+C\ell^{\alpha}(h)\|u\|^2_{\mathbb H^{1+a}(\Omega;\ell)}.\label{eq:est_aprox_sum}
\end{align}
Let $x_i:=x_0=0$ if $i<0$ and $x_i:=x_{N+1}=L$ if $i>N+1$.  Then,
\begin{align}\notag
   \sum_{i\in\inter{0,N}} &\int_{T_i}\int_{T_i}\frac{|u(x)-u(y)|^2}{|x-y|\ell^{1+\alpha}(|x-y|)}dy dx
   \leq \sum_{i\in\inter{0,N}} \int_{x_{i-2}}^{x_{i+3}}\int_{\Omega}\frac{|u(x)-u(y)|^2}{|x-y|\ell^{1+\alpha}(|x-y|)}dy dx\notag\\
   &\leq \sum_{i\in\inter{0,N}} \sum_{j\in \inter{i-2,i+3}}\int_{\Omega_j}\int_{\Omega}\frac{|u(x)-u(y)|^2}{|x-y|\ell^{1+\alpha}(|x-y|)}dy dx\leq 5\int_{\Omega}\int_{\Omega} \frac{|u(x)-u(y)|^2}{|x-y|\ell^{1+\alpha}(|x-y|)}dy dx. \label{eq:est_aprox_omega}
\end{align}
The claim now follows from \eqref{eq:est_aprox_sum} and \eqref{eq:est_aprox_omega}.
\end{proof}

\section{Stability of the error}\label{sec:error}

We begin by recalling the variational formulation associated to our problem: find $u\in\mathbb H(\Omega)$ such that 
\begin{equation}\label{eq:cont_problem}
    \cE_{L}(u,v)=\int_{\Omega}f v \quad \text{ for all } v\in \mathbb H(\Omega)
\end{equation}
with $\cE_{L}(u,v)=\cE(u,v)+B(u,v)$ and where we recall \eqref{eq:bilinear_B}.

Our goal is to approximate the solution $u$ to \eqref{eq:cont_problem} by solving the following discrete problem: find $u_h\in \mathcal V_h$ such that
\begin{equation}\label{eq:discr_problem}
    \cE_{L}(u_h,v)=\int_{\Omega} f v \quad \text{ for all } v\in \mathcal V_h.
\end{equation}

Unlike the case of the fractional Laplacian, we observe that the bilinear form $\cE_{L}(u,u)$ is not coercive in $\mathbb H(\Omega)$ so we cannot employ the classical Cea's lemma. Instead, we shall use the following variant which employs a discrete inf-sup condition. Some of the arguments shown below are well-known and standard in other settings, but we include proofs for completeness.
\begin{lemma}\label{lem:cea_mg}
Assume that there exists a constant $\alpha_0>0$ uniform with respect to $h$ such that
\begin{equation}\label{eq:discr_inf_sup}
    \sup_{v_h\in \mathcal V_h}\frac{\cE_{L}(z_h,v_h)}{\|v_h\|_{\mathbb H(\Omega)}}\geq \alpha_0\|z_h\|_{\mathbb H(\Omega)} \quad\text{ for all } z_h\in\mathcal V_h,
\end{equation}
then there is a unique solution $u_h\in \mathcal V_h$ to \eqref{eq:discr_problem}. Moreover, this solution satisfies
\begin{equation}\label{eq:stability_cea}
    \|u-u_h\|_{\mathbb H(\Omega)}\leq \left(1+\frac{C}{\alpha_0}\right)\inf_{w_h\in\mathcal V_h}\|u-w_h\|_{\mathbb H(\Omega)},
\end{equation}
where $C>0$ is uniform with respect to $h$.
\end{lemma}
\begin{proof}
Let $(\varphi_i)_{i\in \inter{0,N+1}}\subset \mathcal V_h$ be the finite dimensional basis given in \eqref{eq:def_basis}, \eqref{eq:def_basis_ext} of $\mathcal V_h$ with $K:=N+2=\dim(\mathcal V_h)$. Setting $u_h=\sum_{i\in\inter{0,N+1}}u_i \varphi_i$ and $v_h=\sum_{i\in\inter{0,N+1}} v_i \varphi_i$, we note that since $\cE_{L}(\cdot,\cdot)$ is a symmetric bilinear form, the discrete problem \eqref{eq:discr_problem} amounts to solve the linear system $A_h U = F_h,$ where $A_h\in\mathbb R^{K\times K}$ with entries $(A_h)_{ij}=\cE_L(\varphi_i,\varphi_j)$, $F\in\R^{K}$ with entries $(F)_i=\int_{\Omega}f \varphi_i$, and $U=(u_1,\ldots, u_K)^T\in \R^{K}$.

The discrete $\inf$-$\sup$ condition \eqref{eq:discr_inf_sup} readily ensures the non-singularity of the matrix $A_h$, thus $U$ can be uniquely determined and so $u_h$. On the other hand, by \eqref{eq:cont_problem} and \eqref{eq:discr_problem},
\begin{align}\label{eq:ort_cond}
    \cE_{L}(u-u_h,v_h)=0 \quad\text{ for all } v_h\in\mathcal V_h.
\end{align}
Let $w_h\in\mathcal V_h$ be an arbitrary function in $\mathcal V_h$. By \eqref{eq:discr_inf_sup} and \eqref{eq:ort_cond} we have
\begin{equation*}
    \alpha_0\|u_h-w_h\|_{\mathbb H(\Omega)} \leq \sup_{v_h\in\mathcal V_h}\frac{\cE_{L}(u_h-w_h,v_h)}{\|v_h\|_{\mathbb H}}=\sup_{v_h\in\mathcal V_h}\frac{\cE_{L}(u-w_h,v_h)}{\|v_h\|_{\mathbb H}}.
\end{equation*}
Using \cite[Lemma 3.4]{HSS22} and the compact embedding $\mathbb H(\Omega)\hookrightarrow L^2(\Omega)$, we can see that there exists $C>0$ such that $|\cE_{L}(u-w_h,v_h)|\leq C\|u-w_h\|_{\mathbb H(\Omega)}\|v_h\|_{\mathbb H(\Omega)},$ whence
\begin{align}\label{eq:est_stab_w}
    \alpha_0\|u_h-w_h\|_{\mathbb H(\Omega)} \leq C\|u-w_h\|_{\mathbb H(\Omega)} \quad\text{ for all } w_h\in\mathcal V_h.
\end{align}
Finally, triangle inequality and \eqref{eq:est_stab_w} yield that
\begin{align*}
\|u-u_h\|_{\mathbb H(\Omega)}\leq \|u-w_h\|_{\mathbb H(\Omega)}+\|w_h-u_h\|_{\mathbb H(\Omega)}\leq \left(1+\frac{C}{\alpha_0}\right)\|u-w_h\|_{\mathbb H(\Omega)}.
\end{align*}
Taking the infimum for $w_h\in \mathcal V_h$ yields the desired result.
\end{proof}

Lemma \ref{lem:cea_mg} assumes a discrete inf-sup condition \eqref{eq:discr_inf_sup}, which is equivalent to the uniform well-posedness of the discrete problems.  The following lemma gives a sufficient condition for \eqref{eq:discr_inf_sup} to hold, which boils down to the well-posedness of the continuous problem and having a sufficiently fine partition.

\begin{lemma}
Assume that the problem $L_\Delta u=0$ in $\Omega$ with $u=0$ in $\R\backslash  \Omega$ only has the trivial solution.  Then, there are $h_0>0$ and $\alpha_0>0$ such that
\begin{equation*}
    \sup_{v_h\in \mathcal V_h}\frac{\cE_{L}(u_h,v_h)}{\|v_h\|_{\mathbb H(\Omega)}}\geq \alpha_0\|u_h\|_{\mathbb H(\Omega)} \quad\text{ for all } u_h\in\mathcal V_h \text{ and }h\in(0,h_0),
\end{equation*}
where $\alpha_0$ is independent of $h$.
\end{lemma}
\begin{proof}
We argue as in \cite[Theorem 4.2.1]{dem20}. Assume, by contradiction, that for every $n\in\N$ there are $h_n<\frac{1}{n}$ and a sequence $(u_n)_{n\in\N}\in \mathcal V_{h_n}$ such that $\|u_n\|_{\mathbb H(\Omega)}=1$ and
\begin{equation}\label{ccea}
    \sup_{v_n\in \mathcal V_{h_n}}\frac{\cE_{L}(u_n,v_n)}{\|v_n\|_{\mathbb H(\Omega)}}<\frac{1}{n}\qquad \text{ for all }n\in\N.
\end{equation}
Since $(u_n)$ is bounded in $\mathbb H(\Omega)$, then there is $u\in \mathbb H(\Omega)$ such that, passing to a subsequence, $u_n \rightharpoonup u$ weakly in $\mathbb H(\Omega)$. Moreover, since $\mathbb H(\Omega)$ is compactly embedded in $L^2(\Omega)$, then $u_n\to u$ in $L^2(\Omega)$ after passing to a subsequence.  Then,
\begin{align}
    \|u_n-u\|^2_{\mathbb H(\Omega)}
    &=\cE(u_n,u_n-u)-\cE(u,u_n-u)=\cE_L(u_n,u_n-u)-B(u_n,u_n-u)-\cE(u,u_n-u),\label{007}
\end{align}
where $B$ is given by \eqref{eq:bilinear_B}.  In particular, $B$ is a bilinear continuous operator in $L^2(\Omega)$.  Note that $\cE(u,u_n-u)\to 0$ as $n\to \infty$ by weak convergence in $\mathbb H(\Omega)$, $B(u_n,u_n-u)\to 0$ as $n\to \infty$ by strong convergence in $L^2(\Omega)$. By Theorem~\ref{density:thm}, there is a sequence $(\widetilde u_n)$ in $\mathcal V_h$ such that $\widetilde u_n\to u$ in $\mathbb H(\Omega)$ as $n\to \infty$, as a consequence,
\begin{align*}
    \cE_L(u_n,u_n-u)
    &=\cE_L(u_n,u_n-\widetilde u_n)+\cE_L(u_n,\widetilde u_n-u)
    \leq \frac{\cE_L(u_n,u_n-\widetilde u_n)}{\|u_n-\widetilde u_n\|_{\mathbb H(\Omega)}}\|u_n-\widetilde u_n\|_{\mathbb H(\Omega)}+\cE_L(u_n,\widetilde u_n-u) \to 0
\end{align*}
as $n\to \infty$, by \eqref{ccea} and arguing as in \eqref{007}, because $(u_n-\widetilde u_n)_{n\in\N}$ and $(u_n)$ are bounded in $\mathbb H(\Omega)$.  As a consequence, $u_n\to u$ in $\mathbb H(\Omega)$ as $n\to\infty$. Then $\|u\|_{\mathbb H(\Omega)}=1$. Let $w\in \mathbb H(\Omega)\backslash \{0\}$ and let $(w_n)\subset \mathcal V_{h_n}$ such that $w_n\to w$ in $\mathbb H(\Omega)$ as $n\to \infty$ (see Theorem~\ref{density:thm}), then
\begin{align*}
    \frac{\cE_{L}(u,w)}{\|w\|_{\mathbb H(\Omega)}}=\lim_{n\to\infty}\frac{\cE_{L}(u_n,w_n)}{\|w_n\|_{\mathbb H(\Omega)}}\leq \lim_{n\to\infty}\sup_{v_h\in \mathcal V_{h_n}}\frac{\cE_{L}(u_n,v_h)}{\|v_h\|_{\mathbb H(\Omega)}}\leq \lim_{n\to\infty}\frac{1}{n}=0.
\end{align*}
Since $w$ is arbitrary, this implies that $u$ is a nontrivial weak solution of $L_\Delta u=0$ in $\Omega$ with $u=0$ in $\R\backslash  \Omega$, but this contradicts our assumption and the claim follows. 
\end{proof}

\begin{remark}
Note that the assumption that the problem $L_\Delta u=0$ in $\Omega$ with $u=0$ in $\R\backslash  \Omega$ only has the trivial solution is equivalent to i) in Theorem \ref{eq:regularity}.
\end{remark}

We are ready to show our main result, Theorem \ref{main:thm:intro}.

\begin{proof}[Proof of Theorem \ref{main:thm:intro}]
Fix $\alpha\in(0,1)$ as given by \Cref{eq:regularity}. Using \Cref{prop:est_interpolator} with this particular $\alpha$ and estimate \eqref{eq:stability_cea} we get $\|u-u_h\|_{\mathbb H(\Omega)}\leq C\ell^{\alpha}(h)\|u\|_{\mathbb{H}^{1+\alpha}(\Omega;\ell)}$ for some constant $C>0$ uniform with respect to $h$. By \Cref{prop:more_regularity} and the regularity estimate \eqref{eq:regularity_classical} we obtain the desired result.
\end{proof}

\section{The stiffness matrix for the logarithmic Laplacian}\label{sec:stiffness}

Recall that $\Omega:=(0,L)$ for some $L>0$.  We use $A^L_h$ to denote the stiffness matrix associated to the logarithmic Laplacian $L_\Delta.$   The entries of the matrix $A^L_h=(b_{ij})_{i,j=0}^{N+1}$ can be obtained by computing $b_{ij}=\cE_L(\varphi_i,\varphi_j)$ for $i,j=0,\ldots,N+1.$ Instead of doing this computation directly, we obtain a closed formula for $b_{ij}$ by differentiating the entries of the stiffness matrix for the fractional Laplacian (which were explicitly computed in \cite{BH17}) and evaluating them at $s=0$.  To be more precise, let $A^s_h=(b_{ij}^s)_{i,j=0}^{N+1}$ be the stiffness matrix for $(-\Delta)^s$. We show the following.
\begin{lemma}\label{lem:derivative:s:m} For $i,j=0,\ldots,N+1,$ we have that $b_{ij}=\partial_s b_{ij}^s \mid_{s=0}.$ \end{lemma}
\begin{proof}
We claim that 
\begin{align}\label{aux:1}
\lim_{s\to 0}\left\|\frac{(-\Delta)^s \varphi_i-\varphi_i}{s}-L_\Delta \varphi_i    \right\|_{L^p(\R)}=0\qquad \text{for $i=0,\ldots,N+1$ and $1<p<\infty$.}
\end{align}
For $i=1,\ldots,N$, this follows directly from \cite[Theorem 1.1]{CW19}, because $\varphi_i$ is Lipschitz continuous in $\R$.  However, the functions $\phi_0$ and $\phi_{N+1}$ have a jump discontinuity at 0 and at $L$, respectively.  Nevertheless, the proof of \cite[Theorem 1.1]{CW19} can be adapted to this case with obvious changes.  For the reader's convenience we include a brief proof in the appendix (see Lemma~\ref{phi0:conv}). Then, integrating by parts (see Lemma~\ref{ibyp:lem}) and using \eqref{aux:1}, $\partial_s b_{ij}^s\mid_{s=0}=\partial_s\cE_s(\varphi_i,\varphi_j)\mid_{s=0}=
\partial_s\int_{\R} (-\Delta)^s\varphi_i \varphi_j\, dx\mid_{s=0}
=\int_{\R} \partial_s(-\Delta)^s\varphi_i\mid_{s=0} \varphi_j\, dx
=\int_{\R} L_\Delta\varphi_i \varphi_j\, dx
=\cE_L(\varphi_i,\varphi_j)=b_{ij},$ where we used that $\int_{\R}\left(\frac{(-\Delta)^s \varphi_i-\varphi_i}{s}-L_\Delta \varphi_i  \right)\varphi_j\, dx
\leq \|\varphi_j\|_{L^\infty(\Omega)}\int_{\Omega}\left|\frac{(-\Delta)^s \varphi_i-\varphi_i}{s}-L_\Delta \varphi_i  \right|\, dx\to 0$ as $s\to 0^+,$ since $\Omega$ is bounded.
\end{proof}

Using Lemma \ref{lem:derivative:s:m} and an explicit computation of the stiffness matrix for the fractional Laplacian (see Appendix  \ref{ns:sec}), one obtains the stiffness matrix $\mathcal A_{h}^{L}=(b_{ij})\in{\R^{(N+2)\times(N+2)}}$ of the logarithmic Laplacian.

For $j,k,N\in \mathbb N$, $k>2$, let
\begin{align*} \notag
\mathsf{l}_{{j}}:= & \frac{-3 j^3 \log (j)+6 j^2 \log (j)+(j-2)^3 (-\log (j-2))+3
   (j-1)^2 (j-2) \log (j-1)+(j+1)^2 (j-2) \log (j+1)-2}{{6}}, \\
   \mathsf{m}_{ N}:= & \frac{1}{{6}}\left[2 (N-3) N^2 \log (N)-N-(N-1)^3 \log (N-1)-(N+1) ((N-4) N+1) \log (N+1)-3\right], \\
\mathsf{p}:=&  -24\log(3) - 12\log(4) - 16\log(9) + 27\log(16) + \log(144), \\ \notag
\mathsf{q}_k:=&-(k-2)^3\log(k-2) + 4(k-1)^3\log(k-1) - 6 k^3 \log(k) +4(1 + k)^3\log(1 + k) - (2 + k)^3\log(2 + k).
\end{align*}
 Moreover, we use $\psi$ to denote the digamma function and $\gamma=-\psi(1)$.

The coefficients $b_{ij}$ with $i,j\in\inter{0,N+1}$ and $j\geq 1$ are given by
\begin{equation}\label{stiff:LogLap}
b_{ij}= h
\begin{cases}
\displaystyle -\frac{\gamma }{{3}}+\frac{{8}}{9}-\frac{{2}}{3}  \log (h)+\frac{{2}}{3}  \log
   (2)+\frac{{1}}{3}  \psi \left(\frac{1}{2}\right), & \textnormal{$i=0$, $j=0$,} \\[4mm]
\displaystyle -\frac{\gamma }{{6}}+\frac{5}{{18}}-\frac{{1}}{3} \log (h)-\frac{{1}}{3}  \log
   (2)+\frac{1}{{6}} \psi \left(\frac{1}{2}\right), & \textnormal{$i=0$, $j=1$,} \\[4mm]
   -\frac{{1}}{3},
 & \parbox[t]{.3\columnwidth}{$i=0$, $j=2$,}\\[4mm]
 \mathsf{l}_{{j}},
 & \parbox[t]{.3\columnwidth}{$i=0$, $j\in\inter{3,N}$,}\\[4mm]
 \mathsf{m}_{{N}}, &  \parbox[t]{.3\columnwidth}{$i=0$, $j=N+1$,}\\[4mm]
\displaystyle -\frac{\mathsf p}{3}, & \textnormal{$i,j\in\inter{1,N}$ and $j-i=2$,} \\[4mm]
\displaystyle \frac{\mathsf q_{j-i}}{6}, & \textnormal{$i,j\in\inter{1,N}$ and $j-i> 2$,} \\[4mm]
\displaystyle -\frac{\gamma}{6} + \frac{11}{18} - \frac{\log(h)}{3}  + \frac{64\log(2) - 54\log(3)}{12} + \frac{\log(2)}{3}  + \frac{1}{6} \psi(\frac 12), & \textnormal{$i\in\inter{1,N-1}$ and $j=i+1$,} \\[4mm]
\displaystyle -\frac{2}{9}\left[6\log(h) + 3\gamma - 11 + \log(64) - 3\psi(\frac12)\right], & \textnormal{$i\in\inter{1,N}$ and $j=i$.}
\end{cases}
\end{equation}

\section{Numerical evidence and illustrations}\label{sec:numerics}

Trying to characterize a logarithmic convergence rate numerically is a difficult and challenging problem from the numerical perspective. In this section we discuss some partial evidence for the optimality of the logarithmic convergence rate stated in Theorem~\ref{main:thm:intro} and we show numerical approximations of some explicit solutions.

\subsection{Some numerical experiments on the optimality of the convergence rate}

In general, a usual way to observe the optimality of the FEM convergence rate is using the torsion function, namely the solution of the boundary-value problem with right-hand side $f\equiv 1$ (see, among others, \cite[Section 5]{AB17} or \cite[Section 5]{BHS19}), which in many cases has an explicit formula. Unfortunately, no closed expression for the torsion function of the logarithmic Laplacian \eqref{t:p:intro} is known.

Because of this, we propose the following. Let $\Omega:=(-1,1)$ and consider the function $u:\R\to\R$ given by
\begin{align}\label{udef}
 u(x)=\frac{1}{\sqrt{-\ln\left(\frac{1-x^2}{2}\right)}} \chi_{\Omega}(x),
\end{align}
where $\chi_{\Omega}$ denotes the characteristic function in the interval $(-1,1)$.  This choice is justified by the following facts
\begin{itemize}
 \item $u$ has a uniformly bounded continuous logarithmic Laplacian in $(-1,1)$ (as direct calculations show, see \cite[Theorem 2.4]{HSLRS23}). Nonetheless, we remark that $L_{\Delta} u$ does not have a closed formula;
 \item $u$ has the optimal boundary regularity\footnote{This is needed to observe the optimal convergence rate. Indeed, solutions which are smoother typically yield better convergence rates, see, for instance, \cite[Section 5 and Table 2]{AB17} for the case of the fractional Laplacian.} for logarithmic Dirichlet problems (see \cite[Theorem 1.2]{HSLRS23}), namely, it behaves as $\ell^{1/2}(d(x))$.
\end{itemize}

Now, consider $h=\frac{1}{N+1}$ for $N\in\mathbb N$, and let $x_i:=-1+ih$ for $i\in\inter{0,N+1}$ so that $-1=x_0<x_1<\ldots<x_{N}<x_{N+1}=1$. Let $f(x) := \sum_{i=0}^{N+1} [L_\Delta u](x_i)\varphi_i(x)$, $f_i:=\sum_{j=0}^{N+1}L_\Delta u(x_j)\int_\Omega \varphi_i\varphi_j\, dx$, $F=(f_i)$, and let $\alpha:=A_h^{-1}F$,
where $A_h$ is the stiffness matrix for the logarithmic Laplacian given in \eqref{stiff:LogLap}. Let 
\begin{align}\label{vdef}
v_h(x):=\sum_{i=0}^{N+1}\alpha_i \phi_i(x), 
\end{align}
which is a way to approximate\footnote{For instance, the norm $\|L_\Delta u-f\|_{L^2}$ can be approximated numerically, this gives an estimate on $\|u_h-v_h\|_{L^2}$ due to the stability of discrete problems, and then
$\|u-u_h\|_{L^2}\leq \|u-v_h\|_{L^2}+\|v_h-u_h\|_{L^2},$ where $\|u-v_h\|_{L^2}$ can also be estimated numerically.
} $u_h$ (defined in Theorem~\ref{main:thm:intro}), see Figure~\ref{fig:2}. Note that $f_i\approx \int_{\Omega}L_{\Delta}u(x)\phi_i(x)dx$. Since there is no closed formula for $L_{\Delta}u$ we compute $f_i$ with the aid of Mathematica 14.0, where \eqref{LL} is computed and approximated numerically on the mesh grid points $x_i$.

\begin{figure}[htb]
\small
	\centering
	\includegraphics{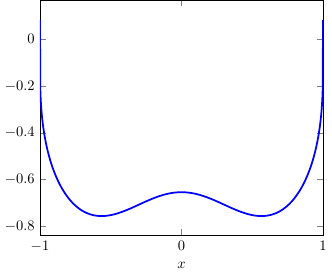}
	\caption{The logarithmic Laplacian of $u(x)$ defined in \eqref{udef}.}
	\label{fig:2}
\end{figure}

\begin{figure}[htb]
\small
	\centering
	\subfloat[]{
	\includegraphics{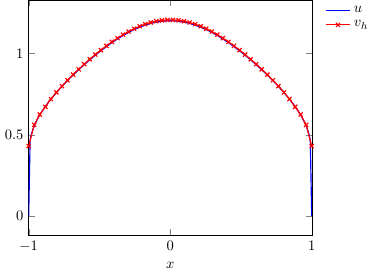}
	}\quad 
	\subfloat[]{
	\includegraphics{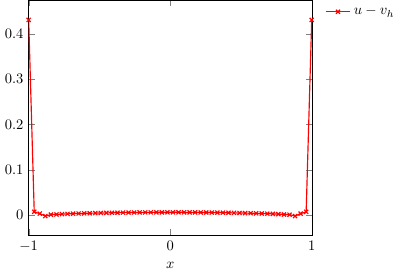}
	}

	\caption{A comparison between $u$ given by \eqref{udef} and $v_h$ given by \eqref{vdef} for $N=50$.}
	\label{fig:2}
\end{figure}

As seen in Figure~\ref{fig:2}, the largest error occurs close to the boundary of the interval, where the function $u$ has a logarithmic behavior.  To quantify this error we compute the following quantities,
\begin{align*}
 a_h&:=\left(\int_{-1}^1 |u-v_h|^2\, dx\right)^\frac{1}{2}\qquad \text{(the $L^2$-norm of the error between $u$ and $v$)},\\
 b_h&:=\left(\int_{K} |u-v_h|^2\, dx\right)^\frac{1}{2}\qquad \text{(a local $L^2$-norm of the error between $u$ and $v$ on a fixed $K\subset\Omega$)},\\
 c_h&:=\sup_{(-1,1)} |u-v_h|\qquad \text{(the $L^\infty$-norm of the error between $u$ and $v$)}.
\end{align*}

We remark that $b_h$ gives information on the error rate in the interior of the interval, far away from the endpoints.

We are particularly interested in exhibiting some evidence for the optimality (or not) of the logarithmic rate in Theorem~\ref{main:thm:intro}.  For this, we compute these quantities for different values of $h$ and describe the behavior of the error curve, which we show in Figure~\ref{fig:3}.

\begin{figure}[htb]
\small
	\centering
	\subfloat[The $L^2$-norm $(a_h)$.]{
	\includegraphics{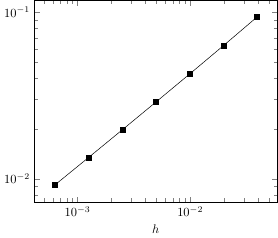}
	}  
	\subfloat[The $L^2_{loc}$-norm $(b_h)$.]{
	\includegraphics{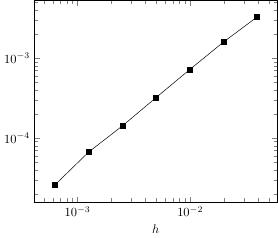}
	}
	\subfloat[The $L^\infty$-norm $(c_h)$.]{
	\includegraphics{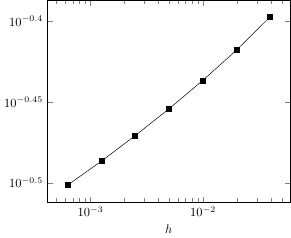}
	}
	\caption{Error curves for $a_h$, $b_h$ (with $K=(-0.9,0.9)$), and $c_h$.}
	\label{fig:3}
\end{figure}

The typical behavior of algorithms with convergence rate of order $h^p$ for some $p>0$ would be a straight line in these plots; in particular, this is the case of the fractional Laplacian, see for instance \cite[Figure 8]{BH17}, \cite[Figure 1]{AB17}, and Table~\ref{table:frac_lap} in the Appendix. The curves in Figure~\ref{fig:3} look almost straight; we present a computation of their slope (with respect to consecutive points) in Table~\ref{table:1}, which shows a slight change. This behavior is robust with respect to the size of the interval.  Below we present the same experiment for the function
$w(x)=\left(
-\ln\left(\frac{1}{2}-2x^2\right)\right)^{-\frac{1}{2}} \chi_{(-\frac{1}{2},\frac{1}{2})}(x)$, see Table~\ref{table:2}.

The discussion above suggests that the optimal convergence rate for the logarithmic Laplacian may not be of power-type when considering $L^\infty$-norms.  The analysis of $L^2$-norms seems to suggest that there might be a possibility to improve the logarithmic convergence rate shown in Theorem~\ref{main:thm:intro} to a polynomial decay rate, at least in the $L^2$-sense. Intuitively, the optimal error rate has an intrinsic connection with the optimal regularity of weak solutions. To enhance the rate (e.g. to a polynomial rate) one would need a more robust regularity theory.  More research in this direction and a better understanding of the regularity of logarithmic problems is required to address these interesting open problems.

\begin{table}[htb]
%\small
\centering
\pgfplotstableread{sol-log_convergence_data_07-02-2025_12h47.org}\data
\pgfplotstableset{
columns={N,h,L2norm,slopeL2,L2locnorm,slopeL2loc,Linfnorm,slopeinf},
%%%
columns/N/.style={
column name=$N$,
},	
%%% 
columns/h/.style={
column name=$h$,
sci,sci zerofill,sci subscript,
precision=2,
column type/.add={}{|}},	
%%%
columns/L2norm/.style={
column name=$L^2$-norm ($a_h$),
sci,sci zerofill,sci subscript,
precision=3},
%%%
columns/slopeL2/.style={
column name=slope,
sci,sci zerofill,sci subscript,
precision=2,
column type/.add={}{|}},	
%%%
columns/L2locnorm/.style={
column name=$L^2_{loc}$-norm ($b_h$),
sci,sci zerofill,sci subscript,
precision=3},
%%%
columns/slopeL2loc/.style={
column name=slope,
%sci,sci zerofill,sci subscript,
precision=2,
column type/.add={}{|}},	
%%%
columns/Linfnorm/.style={
column name=$L_{\infty}$-norm ($c_h$),
sci,sci zerofill,sci subscript,
precision=4,
column type/.add={}{}},
%%%
columns/slopeinf/.style={
column name=slope,
sci,sci zerofill,sci subscript,
precision=2,
column type/.add={}{}},
}
\pgfplotstabletypeset[clear infinite, empty cells with={\ensuremath{-}},
every head row/.style={
before row=\toprule,after row=\midrule},
every last row/.style={
after row=\bottomrule},
]{\data}
\caption{Error data  for $a_h$, $b_h$ (with $K=(-0.9,0.9)$), and $c_h$.}
\label{table:1}
\end{table}
%%% the compact set for the above experiment was (-0.95,0.95)

\begin{table}[htb]
%\small
\centering
\pgfplotstableread{sol-log_convergence_data_07-02-2025_12h53.org}\data
\pgfplotstableset{
columns={N,h,L2norm,slopeL2,L2locnorm,slopeL2loc,Linfnorm,slopeinf},
%%%
columns/N/.style={
column name=$N$,
},	
%%% 
columns/h/.style={
column name=$h$,
sci,sci zerofill,sci subscript,
precision=2,
column type/.add={}{|}},	
%%%
columns/L2norm/.style={
column name=$L^2$-norm,
sci,sci zerofill,sci subscript,
precision=3},
%%%
columns/slopeL2/.style={
column name=slope,
sci,sci zerofill,sci subscript,
precision=2,
column type/.add={}{|}},	
%%%
columns/L2locnorm/.style={
column name=$L^2_{loc}$-norm,
sci,sci zerofill,sci subscript,
precision=3},
%%%
columns/slopeL2loc/.style={
column name=slope,
%sci,sci zerofill,sci subscript,
precision=2,
column type/.add={}{|}},	
%%%
columns/Linfnorm/.style={
column name=$L_{\infty}$-norm,
sci,sci zerofill,sci subscript,
precision=4,
column type/.add={}{}},
%%%
columns/slopeinf/.style={
column name=slope,
sci,sci zerofill,sci subscript,
precision=2,
column type/.add={}{}},
}
\pgfplotstabletypeset[clear infinite, empty cells with={\ensuremath{-}},
every head row/.style={
before row=\toprule,after row=\midrule},
every last row/.style={
after row=\bottomrule},
]{\data}
\caption{Error data for the approximation of $w$. Here we used $K=(-0.4,0.4)$ in the definition of $b_h$.}
\label{table:2}
\end{table}
%%% The the compact set for the L_loc nom was (-0.4,0.4) 

\subsection{Approximation of explicit solutions}

Now, we use three explicit solutions to further illustrate the FEM approximation of solutions to logarithmic Dirichlet problems. For $L>0$, let  $\Omega:=(-L,L)$, $h_\Omega(x):=-\ln(L^2-|x|^2)$, and $\rho_1:=2 \ln (2)+\psi\left(\frac{1}{2}\right)-\gamma\approx -1.15443$, where $\gamma=-\Gamma^{\prime}(1)$ is the Euler-Mascheroni constant and $\psi$ is the Digamma function given by $\psi=\frac{\Gamma^{\prime}}{\Gamma}$.  For $x\in\R$, let 
\begin{align*}
u_1(x):=\chi_{[-L,L]}(x), \qquad
u_2(x):=\chi_{[-L,L]}(x) x, \qquad
u_3(x):=(L^2-x^2)_+=(L^2-x^2)\chi_{[-L,L]}(x).
\end{align*}
Then, for $x\in \Omega$, it is not difficult to check that
\begin{align*}
    L_\Delta u_1(x)&=h_\Omega(x)+\rho_1,\quad
    L_\Delta u_2(x)&=x(2+h_\Omega(x)+\rho_1),\quad
    L_\Delta u_3(x)&=L^2-3x^2+(h_\Omega(x)+\rho_1)(L^2-x^2).
\end{align*}

We remark that the functions $u_1$ and $u_2$ do not satisfy the assumptions of Theorem~\ref{main:thm:intro}, because the right-hand side does not belong to $\mathcal Y(\Omega)$. Nevertheless, the FEM produces an approximation also in these cases, see \Cref{table:u1u2u3} for further information on the quality of approximation.  We shall mention that the function $u_3$ is Lipschitz in $\R$ and thus it exhibits a better convergence rate. Therefore, we cannot use $u_3$ to show the optimality of the convergence rate stated in Theorem~\ref{main:thm:intro}.

\begin{table}[htb!]
%\small
\centering
\pgfplotstableread{sol-log_convergence_data_04-02-2025_19h30.org}\data
\pgfplotstableset{
columns={N,h,L2norm,slopeL2},
%%%
columns/N/.style={
column name=$N$,
},	
%%% 
columns/h/.style={
column name=$h$,
sci,sci zerofill,sci subscript,
precision=2,
column type/.add={}{|}},	
%%%
columns/L2norm/.style={
column name=$L^2$-norm,
sci,sci zerofill,sci subscript,
precision=3},
%%%
columns/slopeL2/.style={
column name=slope,
%sci,sci zerofill,sci subscript,
precision=2,
column type/.add={}{|}},	
%%%
}
\pgfplotstabletypeset[clear infinite, empty cells with={\ensuremath{-}}, 
every head row/.style={before row={ \toprule & & \multicolumn{2}{c}{$u_1(x)$}  \\ },after row=\midrule},
every last row/.style={
after row=\bottomrule},
]{\data}
\pgfplotstableread{sol-log_convergence_data_04-02-2025_19h35.org}\data
\pgfplotstableset{
columns={L2norm,slopeL2},
%%%
columns/L2norm/.style={
column name=$L^2$-norm,
sci,sci zerofill,sci subscript,
precision=3},
%%%
columns/slopeL2/.style={
column name=slope,
%sci,sci zerofill,sci subscript,
precision=2,
column type/.add={}{|}},	
%%%
}
\hspace{-0.35cm}
\pgfplotstabletypeset[clear infinite, empty cells with={\ensuremath{-}},
every head row/.style={before row={ \toprule \multicolumn{2}{c}{$u_2(x)$}  \\ },after row=\midrule},
every last row/.style={
after row=\bottomrule},
]{\data}
\pgfplotstableread{sol-log_convergence_data_04-02-2025_19h36.org}\data
\pgfplotstableset{
columns={L2norm,slopeL2},
%%%
columns/L2norm/.style={
column name=$L^2$-norm,
sci,sci zerofill,sci subscript,
precision=3},
%%%
columns/slopeL2/.style={
column name=slope,
%sci,sci zerofill,sci subscript,
precision=2,
column type/.add={}{}},	
%%%
}
\hspace{-0.35cm}
\pgfplotstabletypeset[clear infinite, empty cells with={\ensuremath{-}},
every head row/.style={before row={ \toprule \multicolumn{2}{c}{$u_3(x)$}  \\ },after row=\midrule},
every last row/.style={
after row=\bottomrule},
]{\data}
\caption{Error data for the approximation of $u_1$, $u_2$ and $u_3$.}
\label{table:u1u2u3}
\end{table}

\section{Approximation of the eigenvalues of the logarithmic Laplacian}\label{eigenvalue approximation:sec}

Standard Fredholm theory implies that if $0$ is not an eigenvalue of $L_\Delta$, then the Dirichlet problem with $L_{\Delta}$ is uniquely solvable for any $f\in L^2(\Omega)$. By the regularity theory, this translates to classical solutions in view of  Theorem~\ref{eq:regularity}. To approximate the eigenvalues of $L_{\Delta}$ in $\Omega$, we consider first the case of Theorem~\ref{eq:regularity}(ii), that is, the case in which $0$ is an eigenvalue of $L_{\Delta}$ in $\Omega$. In this case there is a nontrivial $\phi\in \mathcal X^{\alpha}(\Omega)$, for some $\alpha>0$, satisfying $L_{\Delta}\phi=0$ in  $\Omega$ and $\phi=0$ in $\R\setminus \Omega$ ---we emphasize here that the following argument can be done in any dimension. Note that if we define for $r>0$ a function $\phi_r=\phi(\cdot/r)$, then $L_{\Delta}\phi_r=-2\ln(r)\phi_r$ in $r\Omega$ and $\phi_r=0$ in $\R\setminus r\Omega$, see \cite[Lemma 2.5]{LW21}. Thus having $0$ as an eigenvalue for $\Omega$ gives us the eigenvalue $-2\ln(r)$ in $r\Omega$.

In this way, if we can approximate the value of $L>0$ such that the operator $L_{\Delta}$ has $0$ as an eigenvalue in $\Omega=(-L,L)$, we can use this information to approximate the eigenvalues of $(-\tilde{L},\tilde{L})$ for a general $\tilde{L}>0$. Note that here we consider a symmetric interval centered at $0$ instead of the interval $(0,L)$. To find these values of $L$, we first show that the eigenvalues of the stiffness matrix $\mathcal A_h^L$ actually converge to the eigenvalues of $L_{\Delta}$ and then we use the condition number of the stiffness $\mathcal{A}^L_h$ to find the cases in which there is zero as an eigenvalue. Here, we denote the condition number of a matrix $A$ by
\begin{equation}\label{defi:condition}
\textnormal{cond}(A):=\Big|\frac{\lambda_{\max}(A)}{\lambda_{\min}(A)}\Big|,
\end{equation}
where with $\sigma(A)=\{\lambda\in \R\;:\; \lambda\text{ is an eigenvalue of $A$}\}$ for a matrix $A$ we denote $\lambda_{\max}(A):=\max_{\lambda\in \sigma(A)}|\lambda|$ and $\lambda_{\min}(A):=\min_{\lambda\in \sigma(A)}|\lambda|.$ In this way, $\textnormal{cond}(\mathcal{A}^L_h)$ blows up while moving $L$ from $0$ to infinity every time there is $0$ as an eigenvalue of $\mathcal{A}^L_h$.

\begin{proposition}\label{eigenvalue approximation a}
For $L>0$ the eigenvalues of $L_{\Delta}$ in $(-L,L)$ are given by $\lambda_i=2\ln(L_i/L)$, $i\in \N$, where for each $i\in \N$ the value $L_i>0$ is such that zero is an eigenvalue of $L_{\Delta}$ in $(-L_i,L_i)$.
\end{proposition}
\begin{proof}
We already know that there is an infinite sequence of eigenvalues $\lambda_1<\lambda_2\leq\ldots\lambda_k\to \infty$ for $k\to\infty$ of $L_{\Delta}$ in $\Omega=(-L,L)$ for some fixed $L$, by \cite[Theorem 1.4]{CW19}. By scalings, one can find a sequence of values $L_i>0$, $i\in \N$ satisfying $0<L_1<L_2<\ldots$ with $L_i\to\infty$ for $i\to \infty$ such that $0$ is an eigenvalue of $L_{\Delta}$ in $(-L_i,L_i)$. Scaling again, it follows that the eigenvalues of $L_{\Delta}$ in $(-L,L)$ are given by $\lambda_i=2\ln(L_i/L)$ for $i\in \N$.
\end{proof}

We next show the convergence of the eigenvalues of $\mathcal{A}^L_h$ to eigenvalues of $L_{\Delta}$ for $h\to 0$. Then we present an approximation of $\textnormal{cond}(A^L_h)$ for small enough $h$ and give the approximations of the eigenvalues in $(0,L)$ for general $L>0$.

For the convergence, we cannot work with Theorem~\ref{main:thm:intro}, since it is not clear whether the eigenfunctions belong to the space $\mathcal{Y} (\Omega)$. We first classify the eigenvalues $\{\lambda_k\}_{k\in \N}$ of $L_{\Delta}$ with the Courant-Fischer minimax principle, see e.g. \cite[Proposition 2.3 and Remark 2.4]{FJW22}, that is it holds
\begin{align}\label{CF}
\lambda_k=\inf_{\substack{M\subset \mH(\Omega)\\ \dim(M)=k}} \max_{\substack{u\in M\setminus\{0\}\\ \|u\|_{L^2(\Omega)}=1}} \mathcal{E}_L(u,u),\quad k\in \N.
\end{align}

\begin{proposition}\label{prop:eigen-approx}
Let $N\in \mathbb N$ and $h:=\frac{L}{N+1}>0$. Then
$$
\lambda_{k,h}:=\inf_{\substack{M_h\subset \mathcal{V}_{h}\\ \dim(M_h)=k}} \max_{\substack{u_h\in M_h\setminus\{0\}\\ \|u_h\|_{L^2(\Omega)}=1}} \mathcal{E}_L(u_h,u_h)\to \lambda_k\quad\text{as $h\to 0$.}
$$
Here, $\lambda_{k,h}$ equals to the $k$-th eigenvalue of $\mathcal{A}^L_{h}$.
\end{proposition}
\begin{proof}
By \eqref{CF}, 
\begin{align}\label{1st}
\lambda_k=\inf_{\substack{M\subset \mH(\Omega)\\ \dim(M)=k}} \max_{\substack{u\in M\setminus\{0\}\\ \|u_h\|_{L^2(\Omega)}=1}} \mathcal{E}_L(u,u)\leq \inf_{\substack{M_h\subset \mathcal{V}_{h}\\ \dim(M_h)=k}} \max_{\substack{u_h\in M_h\setminus\{0\}\\ \|u_h\|_{L^2(\Omega)}=1}} \mathcal{E}_L(u_h,u_h)=\lambda_{k,h}.
\end{align}
 By \eqref{c:em}, there is $C=C(\Omega)>0$ such that $\|u\|_{L^2(\Omega)}\leq C\|u\|_{\mH(\Omega)}$ for all $u\in \mH(\Omega)$. Let  $\delta\in(0,1)$. Then there is a subspace $M_{\delta}\subset \mH(\Omega)$ with $\dim(M_{\delta})=k$ and some $m\in M_{\delta}$ with $\|m\|_{L^2(\Omega)}=1$ such that
\begin{align}\label{delta:eq}
\max_{\substack{ u\in M_{\delta}\setminus \{0\}\\ \|u\|_{L^2(\Omega)}=1}} \mathcal{E}_L(u,u)=\mathcal{E}_{L}(m,m)\leq \lambda_k+\frac{\delta}{2}. 
\end{align}
Using Theorem~\ref{density:thm} we find $h_0>0$ such that $\|m-m_h\|_{\mH(\Omega)}\leq \frac{\delta}{2C}$ for all $h\in(0,h_0)$ and for some $m_h\in {\mathcal V}_h$. Moreover, $0<\|m_h\|_{L^2(\Omega)}\leq \|m-m_h\|_{L^2(\Omega)}+1\leq \frac{\delta}{2}+1
$. Then, using \eqref{1st} and \eqref{delta:eq},
\begin{align}\notag 
\lambda_{k}\leq \lambda_{k,h}&\leq \|m_h\|_{L^2(\Omega)}^2 \mathcal{E}_L(m_h,m_h)\leq \left( \frac{\delta}{2}+1\right)^2\Big(\mathcal{E}_L(m_h-m,m_h+m)+\mathcal{E}_L(m,m)\Big)\notag\\
&\leq \left( \frac{\delta}{2}+1\right)^2\mathcal{E}_L(m_h-m,m_h+m)
+\left( \frac{\delta}{2}+1\right)^2\big(\lambda_k+\delta\big).\label{lambda-k-upperbound}
\end{align}
By the definition of $\mathcal{E}_L$ (see \eqref{cEL:def}) and the Cauchy-Schwarz inequality,
\begin{align}\notag 
\mathcal{E}_L(m_h-m,m_h+m)&\leq \mathcal{E}(m_h-m,m_h+m)+\|m_h-m\|_{L^1(\R)}\|m_h+m\|_{L^1(\R)}\\ \notag
&\quad +\rho_1\|m_h-m\|_{L^2(\R)}\|m_h+m\|_{L^2(\R)}\notag \\
&\leq \|m_h-m\|_{\mH(\Omega)}\Big(\|m_h\|_{\mH(\Omega)}+\|m\|_{\mH(\Omega)}\Big)\notag \\ 
&\quad +\Big(|\Omega|^2+\rho_1\Big)\|m_h-m\|_{L^2(\R)}\Big(\|m_h\|_{L^2(\R)}+\|m\|_{L^2(\R)}\Big)\notag\\
&\leq \frac{\delta}{2C}\Big(\|m_h\|_{\mH(\Omega)}+\|m\|_{\mH(\Omega)}+3C\big(|\Omega|^2+\rho_1\big)\Big).\label{lambda-k-upperbound2}
\end{align}
Since $\|m_h\|_{\mH(\Omega)}\to \|m\|_{\mH(\Omega)}$ as $h\to 0$, we can assume that
$\|m_h\|_{\mH(\Omega)}\leq \|m\|_{\mH(\Omega)}+1$ by making $h_0$ smaller, if necessary. Then, from \eqref{lambda-k-upperbound} and \eqref{lambda-k-upperbound2}, we have $\lambda_k\leq \lambda_{k,h}
\leq 
\delta\frac{(\delta+2)^2}{8C}\Big(2\|m\|_{\mH(\Omega)}+1+3C\big(|\Omega|^2+\rho_1\big)\Big)
+
\left( \frac{\delta}{2}+1\right)^2\big(\lambda_k+\delta\big).$ Sending $\delta\to 0$, the right-hand side converges to $\lambda_k$ and this shows the first part of the statement.

For the last part, note that, since the $\phi_i$ ($i\in\inter{0,N+1}$) form a basis of $\mathcal{V}_h$, the dimension of $\mathcal{V}_{h}$ is $N+2$ and, for any $u_h\in \mathcal{V}_{h}$, we can find $r_k\in \R$ with $k\in \inter{0,N+1}$ such that $u_h=\sum_{k=0}^{N+1}r_k\phi_k$ and then $\mathcal{E}_L(u_h,u_h)=\sum_{k,i=0}^{N+1}r_kr_i\mathcal{E}_L(\phi_i,\phi_j)= \mathcal{A}^L_h r\cdot r.$ From here, it is easy to see that $\lambda_{k,h}$ denotes the $k$-the eigenvalue of $\mathcal{A}^L_{h}$.
\end{proof}

\begin{remark}
We emphasize that the scaling of the logarithmic Laplacian behaves the same in all dimensions. Due to this, and by following the proof of Proposition~\ref{prop:eigen-approx} closely, the approach to estimate the eigenvalues of $L_{\Delta}$ in an arbitrary domain $\Omega\subset\R^N$ can be done in an equivalent way, once the stiffness matrix is calculated and Theorem~\ref{density:thm} is verified for the corresponding approximation space $\mathcal{V}_h$.
\end{remark}

We close this section with explicit numerical approximations. In \Cref{Fig:eigen}, we show the numerical approximation of the first and second eigenfunctions of \( L_{\Delta} \) for two different intervals. As noted in \cite[Theorem 1.4]{CW19} (see also \cite{FJW22}), the first eigenfunction is unique and can be chosen to be positive, which agrees with the experiment. On the other hand, very few properties are known for the second eigenfunction. From the figure, we observe that it changes sign exactly once and it is antisymmetric, suggesting an interesting direction for further research.

\begin{figure}[htb]
	\centering
	\subfloat[First eigenfunction]{
	\includegraphics{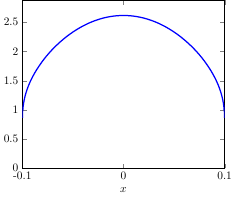}
	}\quad
	\subfloat[Second eigenfunction]{
	\includegraphics{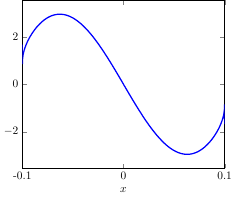}
	} \\
	\subfloat[First eigenfunction]{
	\includegraphics{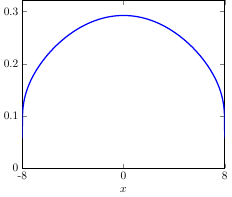}
	}\quad
	\subfloat[Second eigenfunction]{
	\includegraphics{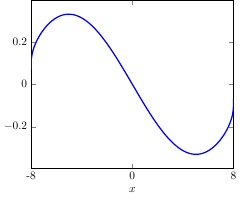}
	} 
	\caption{Numerical approximation of eigenfunctions of \( L_{\Delta} \) for  $L = 0.1$ (top) and $L = 8$ (bottom).
}
	\label{Fig:eigen}
\end{figure}

In Figure~\ref{L vs condition} the condition of $\mathcal{A}^L_h$ is plotted for a fixed $N$ and varying $L$. The singularities then indicated the relevant areas in which $L$ can be chosen to have $0$ as an eigenvalue in $(-L,L)$. To approximate the exact values $L_i$ for $i=1,\ldots,6$ we use a golden-section search to find local maxima.

%%
%%
% reference for the Golden-section search:
%  book \emph{Numerical Recipes} by Press, William H., Teukolsky, Saul A., Vetterling, William T. and Flannery, Brian P. (2007)
%
%%
%%

\begin{center}
\begin{tabular}{c|c|ccccccc}
$L_i$ & interval & $N=2^4$ & $N=2^5$ & $N=2^6$ & $N=2^7$ & $N=2^8$ & $N=2^9$ & $N=2^{10}$\\
\hline
$L_1$ & $[0.6,0.8]$ & $0.7153$ &$0.7119 $ &$0.7103$ & $0.7095$ & $0.7092$ & $ 0.7090$ &$ 0.7090$\\
$L_2$ & $[2.3,2.5]$ & $2.3982$ &$2.3906$ &$2.3844$ & $2.3811$ &$2.3796$ &$2.3790$&$2.3787$\\
$L_3$ & $[3.8,4.0]$ & $3.9249  $ &$3.9279 $ &$3.9205 $ & $3.9153$ &$3.9127 $ &$3.9115$&$3.9110$\\
$L_4$ & $[5.4,5.6]$ & $5.5053$ &$5.5272 $ &$5.5210$ & $5.5141 $ &$ 5.5103 $ &$5.5085$&$5.5077$\\
$L_5$ & $[7.0,7.2]$ & $7.0338$ &$ 7.0772 $ &$7.0766 $ & $7.0690$ &$ 7.0642  $ &$7.0619$&$7.0608$\\
$L_6$ & $[8.5,8.7]$ & $8.6133$ &$8.6573 $ &$8.6640$ & $8.6563$ &$8.6504$ &$8.6474$&$8.6461$
\end{tabular} 
\end{center}           

\begin{proof}[Proof of Proposition~\ref{eigenvalue approximation}]
This follows immediately from Proposition~\ref{eigenvalue approximation a}, the scaling properties of $L_{\Delta}$, and the above table.
\end{proof}

\begin{figure}[htbp]
	\centering
	\includegraphics{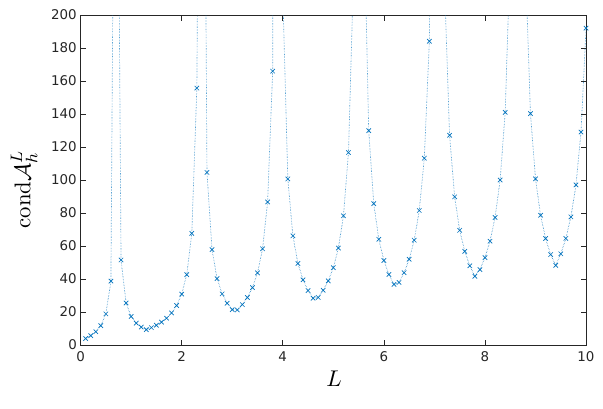}
	\caption{The condition of $\mathcal{A}^L_h$ for $L\in(0.1,10)$ with $N=2000$.}
	\label{L vs condition}
\end{figure}

Lastly, \Cref{Fig:blow} illustrates the behavior of the torsion function as \( L \to L_1\sim 0.7090 \). Since this corresponds to an eigenvalue approaching zero, we expect the torsion function to uniformly blow up in this limit (due to nonexistence), which is indeed observed in Figure \ref{Fig:blow}. The torsion function is positive in the limit from the left and negative in the limit from the right.

\begin{figure}[htb]
	\centering
	\subfloat[$L=0.7040$]{
	\includegraphics{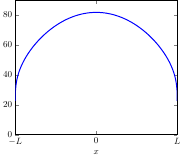}
	} 
	\subfloat[$L=0.7089$]{
	\includegraphics{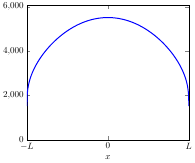}
	} 
	\subfloat[$L=0.7091$]{
	\includegraphics{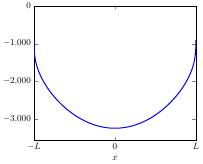}
	} 
	\subfloat[$L=0.714$]{
	\includegraphics{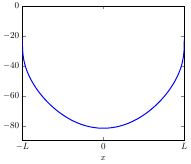}
	} 

	\caption{Blow up behavior of the torsion function as $L\to L_1\sim 0.7090$ from the left and from the right. In these experiments $N=2^{10}$.
}
	\label{Fig:blow}
\end{figure}

 \appendix

\section{Some auxiliary lemmas}
\begin{lemma}\label{lem:over}
Let $R>0$ and $\alpha\in(0,1)$. Then there is $C=C(R,\alpha)>0$ such that $\int_0^{R} \frac{\ell^{1+\alpha}(\rho)}{\rho}d\rho\leq C.$
\end{lemma}
\begin{proof}
Let $r\in(0,0.1)$ be fixed and assume that $R>0.1$. We split the integral into two parts, that is $\int_{0}^{R}\frac{\ell^{1+\alpha}(\rho)}{\rho}d\rho=\int_{0}^{r} \frac{\ell^{1+\alpha}(\rho)}{\rho}d\rho + \int_{r}^{R}\frac{\ell^{1+\alpha}(\rho)}{\rho}d\rho.$ Using that $d\rho/\rho=d(\ln \rho)$ and by definition of the function $\ell$, we have
\begin{align*}
    \int_{0}^{R}\frac{\ell^{1+\alpha}(\rho)}{\rho}d\rho &=-\int_{0}^{r} \frac{d|\ln(\rho)|}{|\ln(\rho)|^{1+\alpha}} + \int_{r}^{R}\frac{\ell^{1+\alpha}(\rho)}{\rho}d\rho 
    \\ &\leq -\int_{0}^{r} \frac{d|\ln(\rho)|}{|\ln(\rho)|^{1+\alpha}} + \int_{r}^{R}\frac{|\ln(r)|^{-1-\alpha}}{\rho}d\rho=\frac{\ell^\alpha(r)}{\alpha}+\ln(r)^{-1-\alpha}\ln(R/r):=C(R,\alpha).
\end{align*}
The case $R\leq 0.1$ follows similarly. This ends the proof. 
\end{proof}
%
% \begin{lemma}\label{lem:ellbeta}
% For every $R>0$ and $\beta>0$ there is $C=C(R,\beta)>0$ such that $\int_0^R \ell^{-\beta}(\rho)\, d\rho<C.$
% \end{lemma}
% \begin{proof}
%  It is well known that the Gamma function has the following integral representation $
%      \Gamma(z)=\int_0^1\left(\ln \frac{1}{t}\right)^{z-1}\, dt$ for $z>0.$ Then, using \eqref{ell:def}, $\int_0^R \ell^{-\beta}(\rho)\,d\rho
% =
% \int_0^{\rho_0} |\ln(x)|^{\beta}\,dx
% +
% \frac{R}{\ell^{\beta}(\rho_0)}
% \leq
% \int_0^{1} |\ln(x)|^{\beta}\,dx
% +
% \frac{R}{\ell^{\beta}(\rho_0)}=\Gamma(1+\beta)+\ell^{-\beta}(\rho_0)R.$
% \end{proof}
%

\begin{lemma}\label{lem:ellbeta}
For every $\beta>0$,
\begin{align*}
\int_0^h \ell^{-\beta}(\rho)\, d\rho = O\left(\frac{h}{\ell^{\beta}(h)}\right)\qquad \text{ as $h\to 0$.}
\end{align*}
In particular, for every $R>0$ there is $C=C(R,\beta)>0$ such that $\int_0^R \ell^{-\beta}(\rho)\, d\rho<C.$
\end{lemma}
\begin{proof}
Observe that, for $h$ small enough and $\beta>0$, using a change of variables $\rho=e^{-t}$ ($t=-\ln \rho,$ $d\rho=-e^{-t} dt$),
\begin{align*}
\int_0^h \ell^{-\beta}(\rho)\, d\rho
=\int_0^h -\ln^{\beta}(\rho)\, d\rho
=\int_{-\ln h}^\infty t^{\beta}e^{-t}\, dt=\Gamma(\beta+1,-\ln h),
\end{align*}
where $\Gamma(\beta+1,-\ln h)$ denotes the (upper) incomplete Gamma function. It is known (see, for instance, \cite{Tem75}) that $\Gamma(s,x)= O\left(x^{s-1}e^{-x}\right)$ as $h\to 0$, and the first claim follows. Hence, if $F(R):=\int_0^R \ell^{-\beta}(\rho)\, d\rho$, then $F$ is continuous in $[0,\infty)$ and $F(0)=0$, from which the second claim easily follows.
\end{proof}

\begin{lemma}\label{lem:int:bds}
Let $\alpha>0$ and $\Omega=(0,L)$ for some $L>0$. There is $C=C(L,\alpha)>0$ such that 
\begin{align*}
\int_{\Omega}\int_{\Omega}\frac{\ell^{1+\alpha}(|x-y|)}{|x-y|\ell^{2+2\alpha}(d(x,y))}\,dy\, dx+ \int_{\Omega}\ell^{-1+\alpha}(d(x))\, dy\,dx<C.
\end{align*}
\end{lemma}
\begin{proof}
Recall that $d(x):=\dist(x,\partial \Omega)$ and $d(x,y)=\min(d(x),d(y)).$ Then, $\int_0^L \ell^{-1+\alpha}(d(x))\,dx\leq \int_0^L \ell^{-1+\alpha}(x)\,dx+\int_0^L \ell^{-1+\alpha}(L-x)\,dx=:C_1<\infty,$ by Lemma~\ref{lem:ellbeta}. On the other hand, by Lemmas~\ref{lem:over} and~\ref{lem:ellbeta},
\begin{align*}
\int_0^L\int_0^L & \frac{\ell^{1+\alpha}(|x-y|)}{|x-y|\ell^{2+2\alpha}(d(x,y))}\,dy\, dx
\leq 2\int_0^L\int_0^L\frac{\ell^{1+\alpha}(|x-y|)}{|x-y|\ell^{2+2\alpha}(d(x))}\,dy\, dx\\
&\leq 4\int_0^L\left(\int_0^L\frac{\ell^{1+\alpha}(|x-y|)}{|x-y|}\,dy\right)\frac{1}{\ell^{2+2\alpha}(x)}\, dx
\leq 4\left(\int_0^L\frac{\ell^{1+\alpha}(\rho)}{\rho}\,d\rho\right)\int_0^L\frac{1}{\ell^{2+2\alpha}(x)}\, dx=:C_2<\infty.
\end{align*}
\end{proof}

The next lemma follows closely the proof of \cite[Theorem 1.1]{CW19}.
\begin{lemma}\label{phi0:conv}
For $1<p<\infty$,
\begin{align*}
\lim_{s\to 0}\left\|\frac{(-\Delta)^s \varphi_0-\varphi_0}{s}-L_\Delta \varphi_0    \right\|_{L^p(\R)}
=\lim_{s\to 0}\left\|\frac{(-\Delta)^s \varphi_{N+1}-\varphi_{N+1}}{s}-L_\Delta \varphi_{N+1}    \right\|_{L^p(\R)}
=0.
\end{align*}
\end{lemma}
\begin{proof}
It suffices to show the claim for $\varphi_0$. Recall that $\Omega=(0,L)$ and $\varphi_0(x)=2(1-\frac{x}{h})\chi_{[0,h]}$ for $h\in(0,1)$ small (with respect to $L$).  Let $R>4.$  For $x \in \mathbb{R}$, we have
$
\left[(-\Delta)^s \varphi_0\right](x)=\int_{\mathbb{R}} \frac{\varphi_0(x)-\varphi_0(x+z)}{|z|^{1+2s}} d z=A_R(s, x)+D_R(s) \varphi_0(x)
$
with $A_R(s, x):=c_{1, s}\left(\int_{B_R} \frac{\varphi_0(x)-\varphi_0(x+z)}{|z|^{1+2 s}} d z-\int_{\mathbb{R} \backslash B_R} \frac{\varphi_0(x+z)}{|z|^{1+2s}} d z\right)$
and $D_R(s):=c_{1, s} \int_{\mathbb{R} \backslash B_R}|z|^{-1-2 s} d z=\frac{c_{1, s}}{s}R^{-2 s}.$ Let $c_{1, s}=s d_1(s)$ with $d_1(s):=\frac{c_{1, s}}{s}=\pi^{-\frac{1}{2}} 2^{2 s} \frac{\Gamma\left(\frac{1}{2}+s\right)}{\Gamma(1-s)}.$  Note that, for $x\in B_R\backslash \{0\}$,
\begin{align*}
\frac{A_R(s, x)}{s} & =d_1(s) \int_{B_R} \frac{\varphi_0(x)-\varphi_0(x+z)}{|z|^{1+2 s}} d z
\rightarrow \tilde{A}_R(x):=\int_{B_R} \frac{\varphi_0(x)-\varphi_0(x+z)}{|z|} d z \quad \text { as } s \rightarrow 0^{+}.
\end{align*}
These integrals are singular at $x=0$, but the singularity is of logarithmic type.  As a consequence, direct computations show that this convergence also holds in $L^p(\R)$ for $1<p<\infty$.  The rest of the argument follows exactly as in \cite[Theorem 1.1]{CW19}.
\end{proof}

\begin{lemma}\label{ibyp:lem}
For $i,j=0,\ldots,N+1$ and $s<\frac{1}{4}$, we have that $\cE_s(\varphi_i,\varphi_j)=\int_{\R}(-\Delta)^s\varphi_i\varphi_j\, dx$ and $\cE_L(\varphi_i,\varphi_j)=\int_{\R}L_\Delta\varphi_i\varphi_j\, dx.$
\end{lemma}
\begin{proof}The proof follows by a standard application of Fubini's theorem.

\end{proof}

\section{Computation of the stiffness matrix of the fractional Laplacian}\label{ns:sec}

In this section we construct the stiffness matrix associated to the fractional Laplacian $(-\Delta)^s$ with $s\in(0,\frac{1}{2})$. We follow the computations done in \cite{BH17}.  In particular, we want to calculate
\begin{align*}
a_{i,j}:=\frac{2}{c_{1,s}}\cE_s(\varphi_i,\varphi_j) \qquad \text{ for }i,j=0,\ldots, N+1,
\end{align*}
where $c_{1,s}=4^s\pi^{-1/2}s(1-s)\frac{\Gamma(1/2+s)}{\Gamma(2-s)}=\frac{\sin (\pi  s) \Gamma (2 s+1)}{\pi }$.

If $i,j\neq 0,N+1$, then these coefficients were obtained in \cite{BH17}, and are given by the following formulas.

If $i,j\in\inter{1,N}$ and $j-i\geq 2$, then
\begin{align*}
	a_{i,j} = - h^{1-2s}\,\frac{4(k+1)^{3-2s} + 4(k-1)^{3-2s}-6k^{3-2s}-(k+2)^{3-2s}-(k-2)^{3-2s}}{2s(1-2s)(1-s)(3-2s)};
\end{align*}
moreover, for $i\in\inter{1,N-1}$,
\begin{align*}
	a_{i,i+1} =	\displaystyle h^{1-2s}\frac{3^{3-2s}-2^{5-2s}+7}{2s(1-2s)(1-s)(3-2s)},
\end{align*}
and, for  $i\in\inter{1,N}$,
\begin{align*}
	a_{i,i} =
			\displaystyle h^{1-2s}\,\frac{2^{3-2s}-4}{s(1-2s)(1-s)(3-2s)}.
\end{align*}

Let $\mathcal A^s_h=\frac{c_{1,s}}{2}(a_{i,j})_{i,j=0}^{N+1}$. It remains to calculate $a_{0,i}$ for $i\in\inter{0,N+1}$.  By symmetry, we have that $a_{0,0}=a_{N+1,N+1}$, $a_{0,i}=a_{i,0}=a_{N+1-i,N+1}=a_{N+1,N+1-i}$ for $i\in\inter{1,N}$. Therefore, we only consider the following cases.

\subsubsection{Case 1: $a_{0,0}$.}

Note that
	\begin{align*}
	a_{0,0}= & \int_{\RR}\int_{\RR}\frac{(\varphi_0(x)-\varphi_0(y))^2}{|x-y|^{1+2s}}\,dxdy
	= 2\int_{h}^{+\infty}\int_{0}^{h} \ldots\,dxdy  + \int_{0}^{h}\int_{0}^{h} \ldots\,dxdy + 2\int_{-\infty}^{0}\int_{0}^{h} \ldots\,dxdy,
\end{align*}
where

% \subsubsection*{Computation of $R_2$}
\begin{align*}
	R_1 &:= 2\int_{h}^{+\infty}\int_{0}^{h} \frac{\varphi_0^2(x)}{|x-y|^{1+2s}}\,dxdy
	= 2\int_{0}^{h}\left(1-\frac{x}{h}\right)^2\left(\int_{h}^{+\infty} \frac{1}{(y-x)^{1+2s}}\, dy\right)\,dx= \frac{1}{s}\int_{0}^{h}\frac{\left(1-\frac{x}{h}\right)^2}{(h-x)^{2s}}\,dx
	=\frac{h^{1-2 s}}{s(3-2 s)}.
\end{align*}

%\subsubsection*{Computation of $R_4$}
\begin{align*}
	R_2 &
	:= \int_{0}^{h}\int_{0}^{h} \frac{(\varphi_0(x)-\varphi_0(y))^2}{|x-y|^{1+2s}}\,dxdy
	= h^{-2}\int_{0}^{h}\int_{0}^{h} |x-y|^{1-2s}\,dxdy
	=\frac{h^{1-2 s}}{(1-s) (3-2 s)}
\end{align*}

%\subsubsection*{Computation of $R_5$}
\begin{align*}
	R_3 &:= 2\int_{-\infty}^{0}\int_{0}^{h} \frac{\varphi_0^2(x)}{|x-y|^{1+2s}}\,dxdy = 2\int_{0}^{h}\varphi_0^2(x)\left(\int_{-\infty}^{0} \frac{dy}{|x-y|^{1+2s}}\right)\,dx = \frac{1}{s}\int_{0}^{h}\frac{\varphi_0^2(x)}{x^{2s}}\,dx\\
	&=\frac{h^{1-2 s}}{(1-s)(3-2 s) (1-2 s)s}.
\end{align*}
as long as $s<1/2$.

Therefore,
\begin{equation*}
    a_{0,0}=\frac{2h^{1 - 2 s}}{s (3 - 2 s) (1 - 2 s)}.
\end{equation*}

\subsubsection{Case 2: $a_{0,j}$ for $j\in\inter{2,N}$.}

\begin{align*}
	a_{0,j}&=-2 \int_{x_{j-1}}^{x_{j+1}}\int_{0}^{h}\frac{\varphi_0(x)\varphi_j(y)}{|x-y|^{1+2s}}\,dxdy
=-2 \int_{(j-1)h}^{(j+1)h}\int_0^{h}\frac{\left(1-\frac{x}{h}\right)\left(1-\frac{|y-jh|}{h}\right)}{|x-y|^{1+2s}}\,dxdy.
\end{align*}

Let us introduce the following change of variables:
\begin{align*}
	\frac{x}{h}=\hat{x},\;\;\; \frac{y-jh}{h}=\hat{y}.
\end{align*}

Then, rewriting (with some abuse of notations since there is no possibility of confusion) $\hat{x}=x$ and $\hat{y}=y$, we get
\begin{align*}%\label{elem_noint_cv}
	a_{0,j}&=-2h^{1-2s} \int_{-1}^1\int_{0}^1\frac{(1-x)(1-|y|)}{(j+y-x)^{1+2s}}\,dxdy\\
	&=-2h^{1-2s} \left(\int_{0}^1\int_{0}^1\frac{(1-x)(1-y)}{(j+y-x)^{1+2s}}\,dxdy
	+\int_{-1}^0\int_{0}^1\frac{(1-x)(1+y)}{(j+y-x)^{1+2s}}\,dxdy\right)\\
	&=-2h^{1-2s}(B_1+B_3).
\end{align*}
Integrating by parts several times, we have
\begin{align*}
	& B_1 = \frac{1}{4s(1-2s)}\left[2j^{1-2s}-\frac{(j+1)^{2-2s}-(j-1)^{2-2s}}{1-s}-\frac{2j^{3-2s}-(j+1)^{3-2s}-(j-1)^{3-2s}}{(1-s)(3-2s)}\right],\\
	& B_3 = \frac{1}{4s(1-2s)}\left[-2j^{1-2s}+\frac{2j^{2-2s}-2(j-1)^{2-2s}}{1-s}+\frac{2(j-1)^{3-2s}-j^{3-2s}-(j-2)^{3-2s}}{(1-s)(3-2s)}\right].
\end{align*}

Therefore,
\begin{align*}
	a_{0,j} =-\frac{h^{1-2 s} \gamma_j^s }{2(1-2 s) (1-s) s (3- 2s )},
\end{align*}
where
\begin{align}\notag
\gamma_j^s= &-3 j^{3-2 s}+2 (3-2 s) j^{2-2 s}+(2 s-3)
   (j-1)^{2-2 s}+3 (j-1)^{3-2 s} \\ \label{eq:def_gamma_j}
   &-(j-2)^{3-2 s}+(j+1)^{3-2 s}+(2
   s-3) (j+1)^{2-2 s}.
\end{align}

\subsubsection{Case 3: $a_{0,N+1}$.} Since $\Omega=(0,L)$, we have that $x_{N+1}=L=(N+1)h=x_{N+2}$ and $x_{N-1}=L-h=h(N-1)$. Then

\begin{align*}
a_{0,N+1}&=-2 \int_{x_{N}}^{x_{N+1}}\int_{x_0}^{x_1}\frac{\varphi_0(x)\varphi_{N+1}(y)}{|x-y|^{1+2s}}\,dxdy
=-2\int_{Nh}^{(N+1)h}\int_0^{h}\frac{\left(1-\frac{x}{h}\right)\left(\frac{y}{h}-N\right)}{(y-x)^{1+2s}}\,dxdy\\
&=-2h^{1-2s}\int_{0}^{1}\int_0^{1}\frac{\left(1-x\right)y}{(y+N-x)^{1+2s}}\,dxdy,
\end{align*}
then
\begin{align*}
a_{0,N+1} = -\frac{h^{1-2 s} \zeta_N^s }{2s(1-s)(1-2s )(3-2 s)},
\end{align*}
where
\begin{align}\notag
\zeta_N^s:= (N-1)^{-2 s} N^{-2 s} (N+1)^{-2 s}
   &\left[(N+1)^{2 s} \left(2 N^2 (N-1)^{2 s} (N+2 s-3)-(N-1)^3
   N^{2 s}\right) \right. \\  & \label{eq:def_zeta_N} \;
   \left. -(N+1) (N-1)^{2 s} N^{2 s} \left(N^2+4 N (s-1)+4
   s^2-6 s+1\right)\right]
\end{align}

\subsubsection{Case 4: $a_{0,1}$.}

We have
	\begin{align*}
	a_{0,1}= & \int_{\RR}\int_{\RR}\frac{(\varphi_0(x)-\varphi_0(y))(\varphi_{1}(x)-\varphi_{1}(y))}{|x-y|^{1+2s}}\,dxdy
	\\
	= & \int_{h}^{+\infty}\int_{h}^{+\infty} \ldots\,dxdy + 2\int_{h}^{+\infty}\int_{0}^{h} \ldots\,dxdy + 2\int_{h}^{+\infty}\int_{-\infty}^{0} \ldots\,dxdy
	\\
	& + \int_{0}^{h}\int_{0}^{h} \ldots\,dxdy + 2\int_{0}^{h}\int_{-\infty}^{0} \ldots\,dxdy + \int_{-\infty}^{0}\int_{-\infty}^{0} \ldots\,dxdy
	\\
	:= & Q_1 + Q_2 + Q_3 + Q_4 + Q_5 + Q_6.
\end{align*}

Let us now compute the terms $Q_i$, $i=1,\ldots,6$, separately.   Note that $Q_6=Q_3=Q_1 = 0.$ Moreover, by Fubini's theorem,
\begin{align*}
	Q_2 &= 2\int_{h}^{+\infty}\int_{0}^{h} \frac{\varphi_0(x)(\varphi_{1}(x)-\varphi_{1}(y))}{|x-y|^{1+2s}}\,dxdy\\
	&= 2\int_{0}^{h}\varphi_0(x)\varphi_{1}(x)\left(\int_{h}^{+\infty} \frac{dy}{|x-y|^{1+2s}}\right)\,dx - 2\int_{h}^{2h}\int_{0}^{h} \frac{\varphi_0(x)\varphi_{1}(y)}{|x-y|^{1+2s}}\,dxdy
	\\
	&= \frac{1}{s}\int_{0}^{h}\frac{\varphi_0(x)\varphi_{1}(x)}{(h-x)^{2s}}\,dx - 2\int_{h}^{2h}\int_{0}^{h} \frac{\varphi_0(x)\varphi_{1}(y)}{|x-y|^{1+2s}}\,dxdy
	\\
	&= \frac{1}{s}\int_{0}^{h}\frac{\left(1-\frac{x}{h}\right)\left(1-\frac{h-x}{h}\right)}{(h-x)^{2s}}\,dx - 2\int_{h}^{2h}\int_{0}^{h} \frac{\left(1-\frac{x}{h}\right)\left(1-\frac{y-h}{h}\right)}{|x-y|^{1+2s}}\,dxdy=:Q_2^1+Q_2^2,
\end{align*}
where, using a change of variables,
\begin{align*}
Q_2^1&=\frac{h^{1-2s}}{s}\int_0^1 x^{1-2s}(1-x)\,dx
= \frac{h^{1-2s}}{2s(1-s)(3-2s)},\\
Q_2^2 &= -2h^{1-2s}\int_0^1\int_0^1\frac{(1-x)(1-y)}{(y-x+1)^{1+2s}}\,dxdy
= h^{1-2s}\frac{2^{2-2s}+2s-4}{2s(1-s)(3-2s)}.
\end{align*}
Adding the two contributions, we get the following expression for the term $Q_2$
\begin{align*}
	Q_2 = h^{1-2s}\frac{2^{2-2s}+2s-3}{2s(1-s)(3-2s)}.
\end{align*}
Finally,
\begin{align*}
	Q_4 &= \int_{0}^{h}\int_{0}^{h} \frac{(\varphi_0(x)-\varphi_0(y))(\varphi_{1}(x)-\varphi_{1}(y))}{|x-y|^{1+2s}}\,dxdy.
\end{align*}
Note that
\begin{align*}
	(\varphi_0(x)-\varphi_0(y))(\varphi_{1}(x)-\varphi_{1}(y))
	= \left(\frac{y-x}{h}\right)\left(\frac{x-y}{h}\right)
	= -\frac{|x-y|^2}{h^2},
\end{align*}
and the integral becomes
\begin{align*}
	Q_4 &= -\frac{1}{h^2}\int_{0}^{h}\int_{0}^{h} |x-y|^{1-2s}\,dxdy
	= -\frac{h^{1-2s}}{(1-s)(3-2s)}.
\end{align*}

By Fubini's theorem,
\begin{align*}
	Q_5 &= 2\int_{0}^{h}\varphi_0(y)\varphi_{1}(y)\left(\int_{-\infty}^{0} \frac{dx}{|x-y|^{1+2s}}\right)dy = \frac{1}{s}\int_{0}^{h}\frac{\varphi_0(y)\varphi_{1}(y)}{(y)^{2s}}\,dy \\
	&= \frac{1}{s}\int_{0}^{h}\frac{\left(1-\frac{y}{h}\right)\left(1-\frac{h-y}{h}\right)}{y^{2s}}\,dx=\frac{h^{1-2 s}}{2(1-s) s (3-2 s)}.
\end{align*}

Then,
\begin{align*}
a_{0,1}=Q_2+Q_4+Q_5=\frac{\left(2^{2-2 s}-2\right) h^{1-2 s}}{2(3-2 s) (1-s) s}.
\end{align*}

\subsubsection{Conclusion}

The stiffness matrix $\mathcal A_h^s = (b_{i,j}^s)_{i,j=0}^{N+1}\in\mathbb R^{(N+2)\times (N+2)}$ has components
\begin{equation}\label{eq:def_aijs_s}
b_{i,j}^s=\frac{c_{1,s}}{2}\int_{\R}\int_{\R}\frac{(\phi_i(x)-\phi_i(y)(\phi_j(x)-\phi_j(y))}{|x-y|^{1+2s}}dx dy, \quad i,j\in\inter{0,1},
\end{equation}
Recall that $\frac{c_{1,s}}{2}=\frac{\sin (\pi  s) \Gamma (2 s+1)}{2\pi}$.

In view of the symmetry of \eqref{eq:def_aijs_s} and using the basis $(\varphi_i)_{i\in\inter{0,N+1}}$ of shape functions \eqref{eq:def_basis}--\eqref{eq:def_basis_ext}, the coefficients $a_{i,j}=\frac{2}{c_{1,s}}\mathcal E_s(\varphi_i,\varphi_j)$ with $i,j\in\inter{0,N+1}$ and $j\geq i$ are given by
\begin{equation}\label{stiff:frac}
a_{i,j}= \frac{h^{1 - 2 s}}{2s(1-2s)(1-s)(3-2s)}
\begin{cases}
4(1-s) & \textnormal{$i=0$, $j=0$,} \\
\left(2^{2-2 s}-2\right)(1-2s) & \textnormal{$i=0$, $j=1$,} \\
-\gamma_j^s  & \parbox[t]{.3\columnwidth}{$i=0$, $j\in\inter{2,N}$, where $\gamma_j^s$ is defined in \eqref{eq:def_gamma_j},}\\
-\zeta^s_{N}  & \parbox[t]{.3\columnwidth}{$i=0$, $j=N+1$, where $\zeta_N^s$ is defined in \eqref{eq:def_zeta_N},}\\
-\xi_k^s & \textnormal{ $i,j\in\inter{1,N}$ with $k=j-i\geq 2$,} \\
3^{3-2s}-2^{5-2s}+7 & \textnormal{$i\in\inter{1,N-1}$ and $j=i+1$,} \\
4^{2-2s}-8 & \textnormal{$i\in\inter{1,N}$ and $j=i$,}
\end{cases}
\end{equation}
where $k:=j-i$ and $\xi^s_k:=4(k+1)^{3-2s} + 4(k-1)^{3-2s}-6k^{3-2s}-(k+2)^{3-2s}-(k-2)^{3-2s}$.

\section{Convergence rates for fractional problems}

The function
\begin{align*}
 U(x)=\frac{(1-x^2)^s}{\Gamma(1+2s)}\chi_{(-1,1)}
\end{align*}
is the torsion function for the fractional Laplacian, namely, $(-\Delta)^s U(x)=1$ for $x\in (-1,1)$ with $s=0.1$.  Using the stiffness matrix given in \eqref{stiff:frac}, we obtain the following table for comparison purposes (cf. Table~\ref{table:1}).  We mention that these rates are optimal, as shown in \cite{Bor17} (see also \cite[Section 5]{AB17} and \cite[Section 5]{BHS19}).

{\color{white}.}
\begin{table}[htb]
\centering
\pgfplotstableread{sol-log_convergence_data_20-06-2023_16h35.org}\data
\pgfplotstableset{
columns={N,h,L2norm,slopeL2,Hsnorm,slopeHs},
%%%
columns/N/.style={
column name=$N$,
},
%%%
columns/h/.style={
column name=$h$,
sci,sci zerofill,sci subscript,
precision=2,
column type/.add={}{|}},
%%%
columns/L2norm/.style={
column name=$L^2$--norm,
sci,sci zerofill,sci subscript,
precision=3},
%%%
columns/slopeL2/.style={
column name=slope,
%sci,sci zerofill,sci subscript,
precision=2,
column type/.add={}{|}},
%%%
columns/Hsnorm/.style={
column name=$H^s$-norm,
sci,sci zerofill,sci subscript,
precision=3},
%%%
columns/slopeHs/.style={
column name=slope,
%sci,sci zerofill,sci subscript,
precision=2,
column type/.add={}{}},
%%%
}
\pgfplotstabletypeset[clear infinite, empty cells with={\ensuremath{-}},
every head row/.style={
before row=\toprule,after row=\midrule},
every last row/.style={
after row=\bottomrule},
]{\data}
\caption{Convergence data fractional Laplacian.}
\label{table:frac_lap}
\end{table}

\paragraph*{Acknowledgments.} We are grateful to the anonymous referees for their insightful comments, corrections, and valuable suggestions, which significantly enhanced the quality of our paper.

 \bibliographystyle{plain}
 \bibliography{biblio}

\bigskip
\begin{flushleft}
\textbf{Víctor Hernández-Santamaría and Alberto Saldaña}\\
Instituto de Matemáticas\\
Universidad Nacional Autónoma de México\\
Circuito Exterior, Ciudad Universitaria\\
04510 Coyoacán, Ciudad de México, Mexico\\
E-mails: \texttt{victor.santamaria@im.unam.mx, alberto.saldana@im.unam.mx} 
\vspace{.3cm}
\end{flushleft}
\begin{flushleft}
\textbf{Sven Jarohs and Leonard Sinsch}\\
Institut f\"ur Mathematik\\
Goethe-Universit\"at Frankfurt\\
Robert-Mayer-Str. 10\\
D-60629 Frankfurt am Main, Germany  \\
E-mails: \texttt{jarohs@math.uni-frankfurt.de, leo\_s1996@web.de} 
\vspace{.3cm}
\end{flushleft}

\end{document}